\documentclass{article} %
\usepackage{iclr2025_conference,times}
\iclrfinalcopy

\usepackage[utf8]{inputenc} %
\usepackage[T1]{fontenc}    %
\usepackage{hyperref}       %
\usepackage{url}            %
\usepackage{booktabs}       %
\usepackage{amsfonts}       %
\usepackage{nicefrac}       %
\usepackage{microtype}      %
\usepackage{xcolor}         %
\usepackage{amsmath, amsthm, todonotes}
\usepackage{thm-restate}

\theoremstyle{plain}
\begingroup
\newtheorem{theorem}{Theorem}[]
\newtheorem*{theorem*}{Theorem}
\newtheorem*{"theorem"}{``Theorem''}

\newtheorem{lemma}[theorem]{Lemma}
\endgroup

\theoremstyle{definition}
\begingroup
\newtheorem{definition}[theorem]{Definition}
\endgroup

\theoremstyle{remark}
\begingroup
\newtheorem{remark}[theorem]{Remark}
\newtheorem{example}[theorem]{Example}
\endgroup

\newcommand{\N}{\mathbb N}

\newcommand{\R}{\mathbb R} 
\newcommand{\E}{{\mathbb E}}
\renewcommand{\L}{{\mathcal L}}

\newcommand{\M}{{\mathcal M}}

\newcommand{\dist}{{\rm dist}}

\newcommand{\df}{\nabla f(x'_n)}

\newcommand{\LRa} {\Leftrightarrow}
\newcommand{\Ra} {\Rightarrow}

\newcommand{\dt}{\,\mathrm{d}t}

\let \eps = \varepsilon
\DeclareMathOperator*{\argmin}{argmin} 
\newcommand\norm[1]{\left\Vert#1\right\Vert}

\hypersetup{
	colorlinks,
	linkcolor={blue!50!black},
	citecolor={blue!50!black},
	urlcolor={blue!80!black}
}

\newenvironment{pde}{\left\{\begin{array}{rll} } {\end{array}\right.}

\title{Nesterov acceleration in benignly non-convex landscapes}

\author{%
  Kanan Gupta,\; Stephan Wojtowytsh\\%\thanks{also \texttt{}} \\
  Department of Mathematics, University of Pittsburgh\\
  \texttt{kanan.g@pitt.edu, s.woj@pitt.edu}
}
\allowdisplaybreaks

\begin{document}

\maketitle

\begin{abstract}
While momentum-based optimization algorithms are commonly used in the notoriously non-convex optimization problems of deep learning, their analysis has historically been restricted to the convex and strongly convex setting. In this article, we partially close this gap between theory and practice and demonstrate that virtually identical guarantees can be obtained in optimization problems with a `benign' non-convexity. We show that these weaker geometric assumptions are well justified in overparametrized deep learning, at least locally. Variations of this result are obtained for a continuous time model of Nesterov's accelerated gradient descent algorithm (NAG), the classical discrete time version of NAG, and versions of NAG with stochastic gradient estimates with purely additive noise and with noise that exhibits both additive and multiplicative scaling.
\end{abstract}

\section{Introduction}

Accelerated first order methods of optimization are the backbone of modern deep learning. So far, theoretical guarantees that momentum-based methods accelerate over memory-less gradient-based methods have been limited to the setting of convex objective functions. Indeed,  recent work of \citet{yue2023lower} shows that the assumption of convexity cannot be weakened as far as, for instance, the Polyak-Lojasiewicz (PL) condition $\|\nabla f\|^2\geq 2\mu\,(f-\inf f)$, which has been used to great success in the study of gradient descent algorithms {\em without} momentum for instance by \cite{karimi2016linear}.

Optimization problems in deep learning are notoriously non-convex. Initial theoretical efforts focused on approximating the training of very wide neural networks by the parameter optimization in a related linear model: The neural tangent kernel (NTK). \cite{jacot2018neural, weinan2019comparative} show that for randomly initialized parameters, gradient flow and gradient descent trajectories remain uniformly close to those which are optimized by the linearization of the neural network around the law of its initialization. This analysis was extended to momentum-based optimization by \citet{liu2022provable}.

Recall that a $C^2$-smooth function $f$ is convex if and only if its Hessian, $D^2f$, is positive semi-definite, i.e.\ has only non-negative eigenvalues.
Strictly negative (but small) eigenvalues of the Hessian of the loss function have been observed close to the set of global minimizers experimentally by \cite{sagun2017eigenvalues, sagun2018empirical, alain2018negative} and their presence has been explained theoretically by \cite{wojtowytsch2023stochastic}. This poses questions about the use of momentum-based optimizers such as SGD with (heavy ball or Nesterov) momentum or Adam in the training of deep neural networks. {In this work, we show that acceleration can be guaranteed for Nesterov's method under much weaker geometric assumptions than (strong) convexity, in particular for certain objective functions that have non-unique and non-isolated minimizers and whose Hessian may have negative eigenvalues up to a certain size.}

In the remainder of this section, we briefly review how our work fits into the literature. In Section \ref{section setting}, we precisely state the assumptions under which we prove convergence at an accelerated rate and discuss how our work connects to optimization in deep learning. Our main results are presented in Section \ref{section results}, both in discrete and continuous time.
Some technical details are postponed to the appendix.

\subsection{Previous work}\label{section previous}

Gradient-based optimization was first proposed by \citet{cauchy1847methode} in form of the gradient descent algorithm. Over a century later, momentum-based `accelerated' algorithms were introduced by \citet{hestenes1952methods} for convex quadratic functions and by \citet{nesterov_original} for general smooth and convex objective functions. Nesterov's work was generalized to non-smooth convex optimization by \citet{beck2009fast} and to stochastic smooth convex optimization among others by \citet{nemirovski2009robust, shamir2013stochastic, jain2019making, laborde2020lyapunov} for additive noise and by \citet{liu2018mass, even2021continuized, vaswani2019fast, gupta2023achieving} for multiplicatively scaling noise. See also \citep{ghadimi2012optimal, ghadimi2013optimal} for more information on accelerated stochastic gradient methods. {While the heavy ball method is used extensively in deep learning, \cite{lessard2016analysis, goujaud2023provable} prove that it does not generally achieve accelerated convergence for smooth strongly convex functions and may even diverge  -- see however \citep{kassing2024polyak} for positive results under stronger smoothness assumptions}

{Accelerated gradient methods have been studied e.g.\ by \citet{josz2023convergence} under much weaker regularity conditions and weaker geometric conditions than (strong) convexity, namely the Kurdyka-Lojasiewicz (KL) condition. Under those weaker assumptions, it is at best possible to prove convergence to a local minimizer at a non-acclerated rate:} Under the (comparatively weak) Polyak-Lojasiewicz (PL) condition, {a special case of the KL condition,} \citet{yue2023lower} show that it is not possible to obtain an accelerated rate of convergence. A slower linear rate of convergence is established by \citet{apidopoulos2022convergence} in continuous time under the assumption that the objective function $f$ satisfies has an $L$-Lipschitz continuous gradient and satisfies the PL-inequality $2\mu\big(f - \inf f\big) \leq \|\nabla f\|^2$. The rate of convergence is
\[
\sqrt{\mu}\left(\sqrt{L/\mu} - \sqrt{L/\mu-1}\right) = \sqrt L\left(1 - \sqrt{1 - \frac\mu L}\right) \approx \sqrt L\cdot \frac\mu{2L} = \frac\mu{2\,\sqrt L}.
\]
A stable time-discretization can generally be attained with effective step-size $1/\sqrt{L}$ for momentum methods \citep[see][Section 2]{su2016differential}, suggesting convergence at the non-accelerated linear rate $(1-\mu/2L)^k$ in discrete time. To the best of our knowledge, no proof has been given yet.

{There have been several efforts to find a reasonable relaxation of convexity for which accelerated convergence can still be achieved.}
\citet{hinder2020near, fu2023accelerated, wang2023continuized, guminov2023accelerated} consider acceleration under the weaker condition that the objective function is $\gamma$-quasar or $(\gamma,\mu)$-strongly quasar-convex, i.e.\ the inequality
\[
\langle \nabla f(x), x- x^*\rangle \geq \gamma\left(f(x) - f(x^*) + \frac\mu2\,\|x-x^*\|^2\right)
\]
holds for any $x\in\R^d$ and {\em any} minimizer $x^*$ of $f$. Compared to (strong) convexity, it relaxes the condition in two ways: It only considers pairs $(x, x^*)$ rather than general pairs of points $x, y\in\R^d$, and it introduces a factor $\gamma$ into the inequality which may be strictly smaller than one. Still, it has geometric implications which may be too strong in the context of deep learning: In the strongly quasar-convex case, minimizers are unique, and in the quasar-convex case, sub-level sets are star-shaped with respect to any minimizer $x^*$ since $f(tx + (1-t)x^*)$ is monotone increasing on $[0,1]$ (see Lemma \ref{lemma geometric implications}, Appendix \ref{appendix convexity and pl}).

Accelerated rates of convergence were obtained by \citet{necoara2019linear} in discrete time and by \citet{aujol2022convergence} in continuous time under the assumption that the objective function is both convex and quasi-strongly convex, and that it has a unique minimizer.
Their results are generalized by \citet{hermant2024studybehaviournesterovaccelerated} who allow for non-smooth composite optimization and consider $\gamma\in(0,1]$ rather than just $\gamma=1$. Unlike the present work, \citet{hermant2024studybehaviournesterovaccelerated} require the uniqueness of minimizers and only study deterministic optimization. Curiously, despite the difference in settings, they independently find the same lower bound on Hessian eigenvalues that we require (compare e.g.\ Theorem \ref{theorem discrete} and \citep[Theorem 2]{hermant2024studybehaviournesterovaccelerated}).
See also \citep[Table 1]{aujol2024heavyballmomentumnonstrongly} for an overview of theoretical guarantees of acceleration without strong convexity.

In overparametrized deep learning, the set of minimizers of the loss function is a (generally curved) manifold, and tangential motion to the manifold can have important implications on the implicit bias of an algorithm \citep{li2021happens, damian2021label}. Any notion that takes into account {\em all} minimizers is quite rigid for such tasks. A more realistic assumption is the `aiming condition' of \citet{liu2024aiming} that
\[
\langle \nabla f(x), x-\pi(x)\rangle \geq \gamma \big(f(x) - \min f\big)
\]
where $\pi(x)$ is the closest minimizer to the point $x$. This notion enjoys much greater flexibility in terms of the global geometry of $f$. \citet{liu2024aiming} investigate the convergence of gradient flows under the aiming condition, but not that of momentum methods.

\subsection{Our contribution}

Our study can be seen as combining geometric ideas pertaining to $(1,\mu)$-quasar convexity and the aiming condition. Our assumption in \eqref{eq first order convexity wrt pi} is equivalent to quasi-strong convexity \citep{necoara2019linear}, but notably we do not make any additional assumptions about the convexity of the objective function or uniqueness of minimizers unlike prior works. Since we focus on the closest minimizer at any point in the trajectory, this presents a unique challenge compared to analyzing the distance from a fixed minimizer. As the current iterate $x_t$ moves, its projection onto the set of minimizers, $\pi(x_t)$, also moves. This requires a modification of the usual Lyapunov function to account for the movement of $\pi(x_t)$ (or the \textit{tangential movement} of $x_t$ parallel to the set of minimizers) as well as the `drift' in directions where the objective function is negatively curved. To the best of our knowledge, this is the first study which takes into account this tangential movement and obtains accelerated convergence.

In overparametrized learning, the set of minimizers of a loss function is a submanifold of high dimension in a usually much higher-dimensional space. {Unless the manifold of minimizers is a linear space, the Hessian of the loss function is geometrically required to have negative eigenvalues in any neighborhood of the set of a minimizer where the manifold is curved.
Still, accelerated methods in first order optimization have been found to be highly successful in deep learning. 
A common heuristic has been that} as long as the objective function is convex in the direction towards the set of minimizers, small negative eigenvalues in directions parallel to the set of minimizers can safely be ignored: Tangential drift along the set of minimizers {should} not affect the decay of the objective function significantly. We prove that this intuition indeed applies in a continuous time model for gradient descent with momentum (Theorems \ref{theorem continuous strongly convex} and \ref{theorem global convergence}) and for Nesterov's time-stepping scheme (Theorems \ref{theorem discrete}, \ref{theorem stochastic additive} and \ref{theorem decreasing learning rate}) in deterministic optimization and stochastic optimization with bounded noise. With `multiplicative (state-dependent) noise' motivated by overparametrized deep learning, we prove an analogous statement for a modified version of Nesterov's algorithm (Theorem \ref{theorem agnes}).

\section{Setting}\label{section setting}

\subsection{Assumptions}\label{section assumptions}
We always make the following assumptions on the regularity and geometry of the function $f$ and its set of minimizers.

\begin{enumerate}
    \item The objective function $f:\R^d\to\R$ is bounded from below, $C^1$-smooth and its gradient $\nabla f:\R^d\to\R^d$ is locally Lipschitz continuous.
    
    \item The set $\mathcal M = \{x\in\R^d : f(x) = \inf_{z\in\R^d}f(z)\}$ of minimizers of $f$
    is a (non-empty) $k$-dimensional $C^2$-submanifold of $\R^d$ for $k<d$.

    \item There exists an open sub-level set $\mathcal U_\alpha = \{ x\in\R^d : f(x) < \alpha\}$ for some $\alpha>0$ such that for every $x\in \mathcal U_\alpha$ there exists a unique $z\in \mathcal M$ which is closest to $x$. We denote $z$ as $\pi(x)$ and assume that the closest point projection map $\pi : \mathcal U_\alpha \to \mathcal M$ is $C^1$-smooth.

    \item $f$ satisfies the {\em $\mu$-strong aiming condition} (with respect to the closest minimizer) in $\mathcal U_\alpha$, i.e.
    \begin{equation}\label{eq first order convexity wrt pi}
    \nabla f(x) \cdot \big(x-\pi(x)\big) \geq f(x) - f \big( \pi(x)\big) + \frac\mu2\,\|x-\pi(x)\|^2\qquad\forall\ x\in \mathcal U_\alpha.
    \end{equation}
\end{enumerate}

These assumptions are significantly weaker than the assumption of strong convexity, but for instance strong enough to imply a PL inequality with constant $\mu$ (see Appendix \ref{appendix convexity and pl}). As illustrated in Section \ref{section deep learning}, they match many geometric features of overparametrized deep learning. For an analysis in discrete time, we will make stronger quantitative assumptions.

            \subsection{Simple Examples}

We give a number of examples which are covered by these assumptions where $f$ is not merely a $\mu$-strongly convex function. The first example illustrates how subtle the interplay between geometric conditions is, even in one dimension.

\begin{example}\label{example 1d}
    For $\eps, R>0$, consider the function $f(x) = \frac{x^2}2 + \frac\eps2\,x^2\sin(2R\,\log(|x|)$. The function $f$ has a unique global minimizer at $x^*=0$ and its derivative $f'$ is a Lipschitz-continuous function with Lipschitz-constant $1 +\eps\sqrt{1+5R^2+4R^4}$. Furthermore, $f$ has an infinite number of strict local minimizers if $\eps\sqrt{1+R^2}>1$, but satisfies favorable geometric properties under stronger assumptions:

\begin{center}
    \renewcommand{\arraystretch}{1.6}
    \begin{tabular}{r|ccc}
        &PL condition& $\mu$-strongly aiming &$\mu$-strongly convex \hspace{2mm}\\\hline 
        Must be $<1$ & $\eps \sqrt{1+R^2}$ &  $\eps \sqrt{1+4R^2}$ & $\eps\sqrt{1+5R^2 + 4R^4}$\\
        Constant & $(1-\eps \sqrt{1+R^2})^2/({1+\eps})$ & $1- \eps\sqrt{1+4R^2}$ & $1-\eps\sqrt{1+5R^2 + 4R^4}$
    \end{tabular}
\end{center}

Evidently, the geometric conditions and associated constants are quite different if $R\gg 1$.
    See Figure \ref{figure 1d example 2} for an illustration of $f$. Further details for the example and a comparison to less common notions such as quasar-convexity are given in Appendix \ref{appendix 1d example}. {We note that the example exploits the fact that $f$ is $C^{1,1}$- but not $C^2$-smooth: For $C^2$-functions, \citet{rebjock2023fast} prove that the PL condition locally implies strong aiming condition.}

    \begin{figure}
    \centering
    \includegraphics[width=.33\textwidth, clip = true, trim = 1cm 0cm 1cm 1cm]{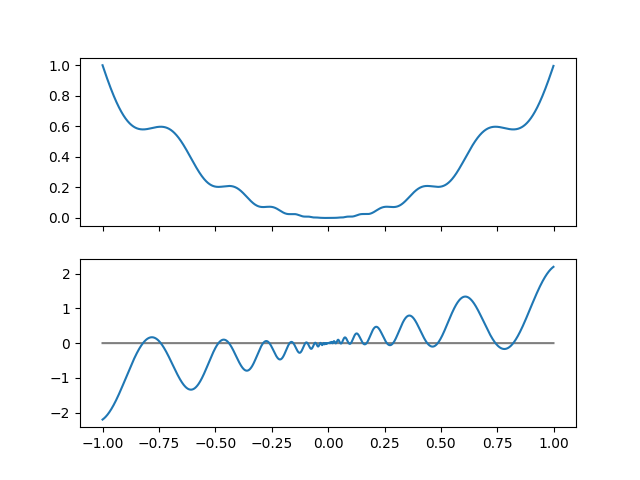}\hfill
    \includegraphics[width=.33\textwidth, clip = true, trim = 1cm 0cm 1cm 1cm]{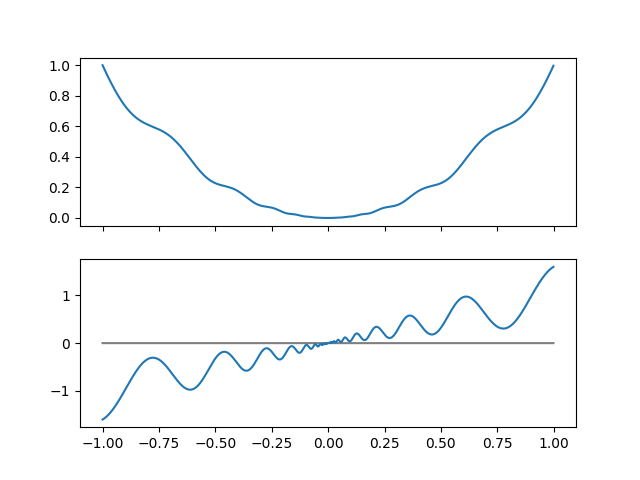}\hfill
    \includegraphics[width=.33\textwidth, clip = true, trim = 1cm 0cm 1cm 1cm]{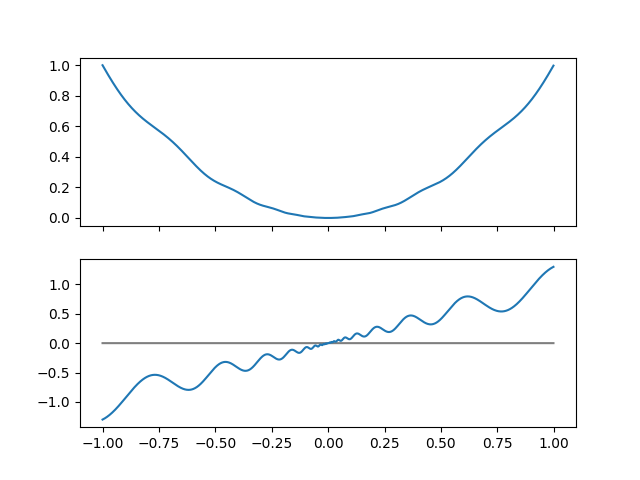}
    \caption{\label{figure 1d example 2}
    We visualize $f$ from Example \ref{example 1d} in the top row and its derivative in the bottom row with $R=2$ and $\eps =0.2$ (left), $\eps=0.1$ (middle) and $\eps = 0.05$ (right). Left: $f$ has many local minimizers as the derivative crosses $0$ an infinite number of times. Middle: $f$ satisfies the PL condition, but not the strong aiming condition. Right: $f$ is strongly aiming (with respect to the unique global minimizer, which implies the PL condition). In all plots, $f$ is non-convex since $f'$ is non-monotone.
    }
\end{figure}
\end{example}

The next example is trivial, but useful to illustrate why tangential movement should not matter.

\begin{example}\label{example product straight}
    Let $\tilde f_1:\R^{d-k} \to[0,\infty)$ be a non-negative $\mu$-strongly convex function such that $\tilde f_1(0) =0$ and let $\tilde f_2:\R^k\to [a,\infty)$ be a continuous function for $a>0$. Define 
    \[
    f:\R^d\to\R, \qquad f(x) = \tilde f_2(x_1,\dots, x_k)\cdot\tilde f_1(x_{k+1}, \dots, x_d).
    \]
    Then $f$ is $a\mu$-strongly aiming $\pi(x) = (x_1, \dots, x_k, 0,\dots,0)$, but not strongly convex since the minimizer is non-unique. Similarly, if $A:\R^k\to\R^{(d-k)\times (d-k)}$ is a function which takes values in the set of symmetric matrices with eigenvalues larger than $\mu$, then 
    \[
    f:\R^d\to\R, \qquad f(x) = \frac12\,(x_{k+1}, \dots, x_d) \,A_{(x_1,\dots, x_k)}\cdot (x_{k+1},\dots, x_d)^T
    \]
    is $\mu$-strongly aiming, but generally non-convex.
\end{example}

\begin{example}\label{example distance squared}
    Let $\mathcal M$ be a compact $C^k$-submanifold of $\R^d$, $k\geq 2$ and $d(x) := \dist(x,\mathcal M)$. Then there exists a `tubular neighborhood' $\mathcal U_\eps = \{x\in\R^d : d(x) < \eps\}$ on which $d$ is $C^k$-smooth and the unique closest point projection $\pi$ is well-defined and $C^{k-1}$-smooth -- see Appendix \ref{appendix differential geometry}.

    Assume that $f:\mathcal U_\eps\to\R$\ is given by $f(x) = \frac\mu2\,d(x)^2$ (and extended arbitrarily to $\R^d\setminus \mathcal U_\eps$). Recall that $\nabla d(x)$ is the unit vector pointing towards the closest point in $\mathcal M$ at all points $x$ where the distance function is smooth, so in particular $\|\nabla d(x)\| = 1$. Thus $\pi(x) = x-d(x)\nabla d(x)$ and
    \begin{align*}
    \nabla f(x) \cdot (x-\pi(x)) &= \mu\,d(x)\nabla d(x) \cdot \big(x - (x-d(x)\nabla d(x))\big) = \mu\,d^2(x) \,\|\nabla d(x)\|^2\\
    &= \frac\mu2\,d^2(x) + \frac\mu2\,d^2(x) =  f(x) - f(\pi(x)) + \frac\mu2\,\|x-\pi(x)\|^2
    \end{align*}
    for $x\in \mathcal U_\eps$, i.e.\ $f$ is $\mu$-strongly aiming. On the other hand, $f$ is not convex unless $\mathcal M$ is. Otherwise, take $x_1, x_2\in\mathcal M$ and $t\in (0,1)$ such that $tx_1 + (1-t)x_2\notin \mathcal M$. Then the map $t\mapsto d^2\big(tx_1 + (1-t)x_2\big)$ attains a maximum inside the interval $(0,1)$, meaning that $d^2$ cannot be convex.
\end{example}

This consideration more generally shows that if the manifold of minimizers $\mathcal M$ of a function $f$ is not perfectly straight, then the objective function cannot be convex -- see also Figure \ref{figure non-convexity}. More precisely:

\begin{lemma}\citep[based on Appendix B]{wojtowytsch2023stochastic}\label{lemma can't be convex}
    Let $f:\R^d\to\R$ be a $C^2$-function and $\mathcal M = \{x\in\R^d : f(x) = \inf f\}$. Assume that $\mathcal M$ is a $k$-dimensional $C^1$-submanifold of $\R^d$, $z\in \mathcal M$, $T_z\mathcal M$ the tangent space at $z$, $r>0$.
    If $\mathcal M\cap B_r(z)$ is not the same set as $(z+ T_z\mathcal M)\cap B_r(z)$, then there exists $x\in B_r(z)$ such that $D^2f(x)$ has a strictly negative eigenvalue.
\end{lemma}

\begin{figure}
    \centering
    \includegraphics[height=3.4cm]{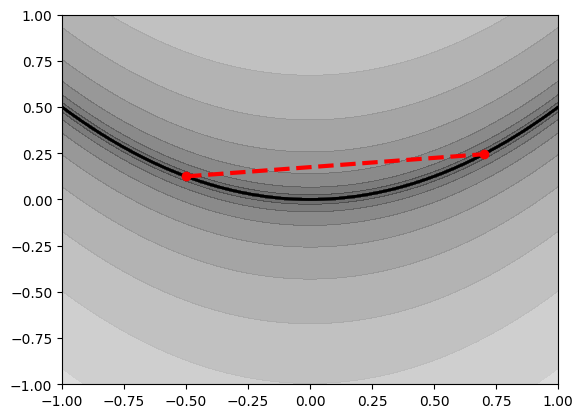}\hfill
    \includegraphics[height=3.4cm]{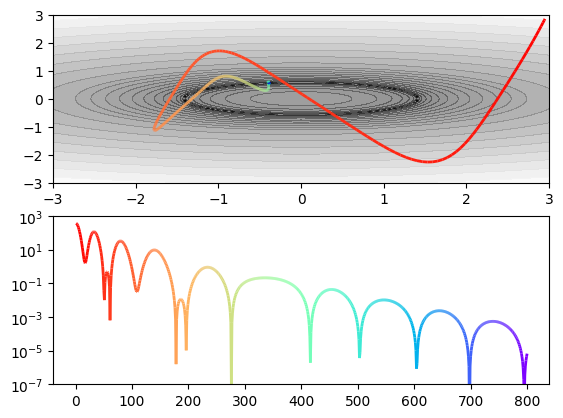}\hfill
    \includegraphics[height=3.4cm]{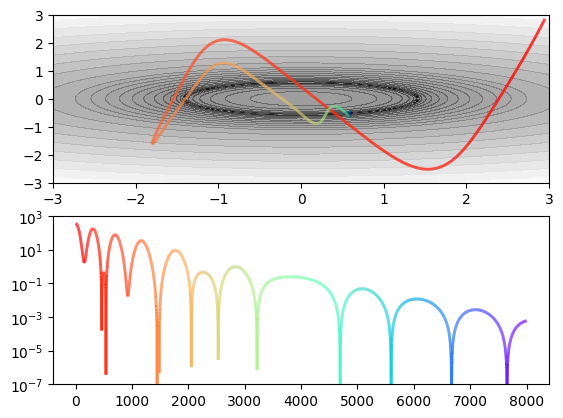}

    \caption{
     \label{figure non-convexity}
    {\bf Left:} The dashed red line connects two minimizers of the function $f$. Along the line, $f$ must achieve an interior local maximum. At this point, the Hessian $D^2f$ cannot be positive definite. {\bf Middle, Right:} Optimization trajectories for Nesterov's method (top) and its associated energy curve (bottom). The selection of limit point may depend crucially on optimization parameters: In the middle plot, we take 800 steps with stepsize $10^{-2}$ while on the right, we take 8,000 steps with stepsize $10^{-3}$ from the same initial point. The decay of $f(x_t)$ is similar for both trajectories, but the limit points on the manifold of minimizers are far apart. The objective function is $f(x,y) = (x^2/2 + 3y^2-1)^2$.
    }
\end{figure}

\subsection{\texorpdfstring{{Connection} to Deep Learning}{Connection to Deep Learning}}\label{section deep learning}

An important class of objective functions are those which combine the geometric features of Examples \ref{example product straight} and \ref{example distance squared}. Such functions can be seen as geometric prototypes for loss functions in {overparametrized regression problems, such as in} deep learning.
Namely, consider a parametrized function class $h:\R^p\times \R^d\to \R$ of weights $w\in \R^p$ and data $x\in\R^d$ (e.g.\ a neural network) and the mean squared error (MSE)  loss function
\begin{equation}\label{eq machine learning loss}
L_y:\R^p\to [0,\infty), \qquad L_y(w) = \frac1{2n}\sum_{i=1}^n \big(h(w, x_i) - y_i\big)^2, \qquad y = (y_1,\dots, y_n)\in\R^n.
\end{equation}
If $h$ is sufficiently smooth in $w$ and for every vector $y\in\R^n$, there exists $w_y\in\R^p$ such that $L_y(w_y) =0$, then \citet{cooper2021global} showed that for Lebesgue-almost all $y\in\R^n$, the set $\mathcal M_y = \{w\in\R^p : L_y(w) = 0\}$ is a $p-n$-dimensional submanifold of $\R^p$. Essentially, the solution set of $n$ equations $h(w,x_i) = y_i$ in $p$ variables is $p-n$-dimensional, much like when $h$ is linear in $w$.

\citet{cooper2021global} demonstrates that the expressivity and smoothness assumptions provably apply to parametrized function classes $h(w,x)$ of sufficiently wide neural networks with analytic activation function such as tanh or sigmoid. \cite{cooper2021global}'s proof involves the regular value theorem and Sard's theorem to show that all gradients $\nabla h_w(w, x_i)$ are linearly independent on almost every level set $\mathcal M_y$. As a byproduct, this implies that the Hessian of the loss function
\begin{align*}
    D^2 L(w) &= \frac1n\sum_{i=1}^n \bigg(\underbrace{\big(h(w, x_i) - y_i\big)}_{=0}D^2_w h(w, x_i) + \nabla_wh(w, x_i) \otimes \nabla_w h(w, x_i)\bigg)
\end{align*}
has full rank $n$ at every $x\in \mathcal M_y$ (for almost every $y$). Thus, all $n$ eigenvalues in direction orthogonal to $\mathcal M$ are non-zero. We prove the following in Appendix \ref{appendix differential geometry}. {The same connection has been made e.g.\ by \citet{rebjock2023fast} for $C^2$-functions and overparametrized regression problems.}

\begin{restatable}{lemma}{application}\label{lemma application}
    Assume that $f:\R^d\to\R$ is $C^2$-smooth and that $\mathcal M = \{x\in\R^d: f(x) = \inf f\}$ is a closed $k$-dimensional $C^2$-submanifold of $\R^d$ (i.e.\ compact and without boundary). If $D^2f(x)$ has rank $d-k$ everywhere on $\mathcal M$, then there exist $\mu, \alpha>0$ such that there exists a $C^1$-smooth closest point projection $\pi:U_\alpha\to\R$ with $U_\alpha = \{x : f(x)<\alpha\}$ %
    and $f$ is $\mu$-strongly aiming with $\pi$.
\end{restatable}

In particular, $f, \mathcal M$ meet all conditions in Section \ref{section assumptions}. Note that by Lemma \ref{lemma can't be convex}, the loss function $f$ cannot be convex since a compact manifold cannot be perfectly straight everywhere.
Lemma \ref{lemma application} does not apply to networks with the non-smooth ReLU activation $\sigma(z) = \max\{z,0\}$, also here minimizers cannot be isolated due to the continuous scaling symmetry $\sigma(z) = \lambda^{-1}\sigma(\lambda z)$ for $\lambda>0$.

\citet[Theorem 2.6]{wojtowytsch2023stochastic} shows that the assumption that $\mathcal M$ is compact is a simplification and generally does not apply in deep learning. Local versions of Lemma \ref{lemma application} could be proved with $\mu, \alpha$ which are positive functions on the manifold, but not necessarily bounded away from zero. Naturally, this suffices in all cases where we provably remain in a local neighborhood in the course of optimization. We eschew this greater generality for the sake of geometric clarity and commit the pervasive sin of optimization theory for deep learning: We make global assumptions which can only be guaranteed locally. For a further comparison of geometric conditions in optimization and deep learning, see also Appendix \ref{appendix convexity and pl}. 
\begin{restatable}{figure}{figureloss}
    \centering
    \includegraphics[clip = true, trim = 1.2cm 8mm 1.5cm .8cm, width = .33\textwidth]{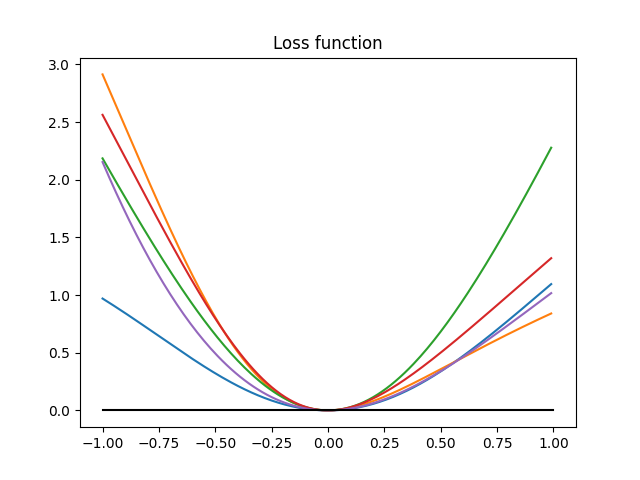}\hfill
    \includegraphics[clip = true, trim = 1.2cm 8mm 1.5cm .8cm, width = .33\textwidth]{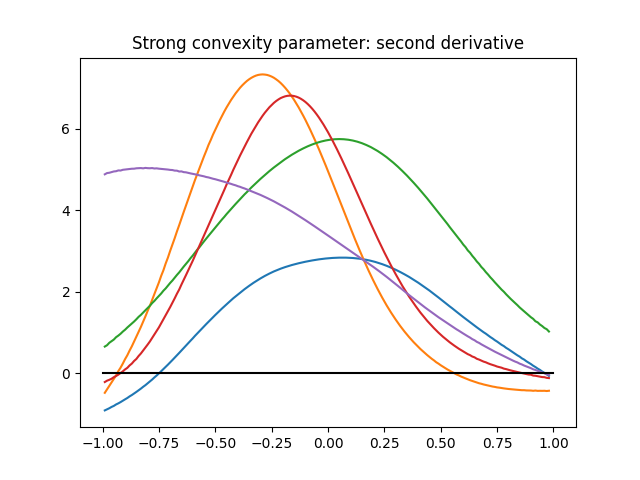}\hfill
    \includegraphics[clip = true, trim = 1.2cm 8mm 1.5cm .8cm, width = .33\textwidth]{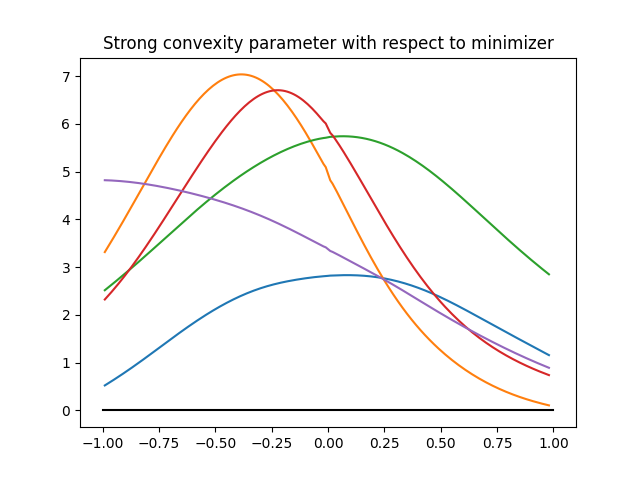}

    \caption{Convexity analysis of \( \phi(t) = L(w+tg) \) for \( w \) near global minimizers of a loss function \( L \) and \( g = \nabla L(w) / \|\nabla L(w)\| \). Left: \( \phi(t) \), middle: second derivative of \( \phi \), right: estimated strong aiming parameter \(\mu\) for \( t \in [-1,1] \). Evidently, $\phi$ is strongly convex in a neighborhood of the  minimizers. Strong aiming condition yields consistently larger constants than the strong convexity parameter obtained from second derivatives. Different colors correspond to different random initializations.}
    \label{figure loss function convexity}
\end{restatable}

We illustrate in Figure \ref{figure loss function convexity} that our assumptions are locally reasonable in deep learning. We trained a fully connected neural network (with 10 layers, width 35, tanh activation) to fit labels \( y_i \) at 100 randomly generated datapoints \( x_i \in \mathbb{R}^{12} \). The small dataset size allowed us to use the exact gradient and loss function instead of stochastic approximations, for a better exploration of the loss landscape. Since the closest minimizer is generally unknown, we use the gradient as a proxy and examine the convexity of \( \phi(t) = L(w+tg) \) for \( w \) very close to the set of global minimizers of the loss function $L$ as in \eqref{eq machine learning loss} and \( g = \nabla L(w) / \|\nabla L(w)\| \). Labels were generated using a randomly initialized teacher network (with 7 layers and width 20). Student networks were trained for 10,000 epochs using stochastic gradient descent with Nesterov momentum, with learning rate $\eta=0.005$ and momentum $\rho=0.99$. Final training loss ranged between \(10^{-12}\) and \(10^{-9}\) across the five runs. Second derivatives were approximated using second order difference quotients $\phi''(t) \approx \frac{\phi(t+h) - 2\phi(t) + \phi(t-h)}{h^2}$ for $h=0.01$.
Similarly, the strong aiming parameter with respect to the global minimizer was estimated by
$2\frac{\phi'(t)t - \phi(t) + \inf \phi}{t^2}$ where $\phi'(t)$ was estimated as $\frac{\phi(t+h) - \phi(t-h)}{2h}.$

Due to the inherent randomness of training, geometric behaviors varied. Generally, the loss function was convex along the gradient direction near minimizers, but neighborhood size, magnitude of the second derivative, and steepness of the loss function varied. Some runs exhibited convex but not strongly convex behavior, while others failed to reach zero loss along the line \( w + tg/\|g\| \). The gradient is an imperfect approximation of the minimizer's direction, and rounding errors may occur when it is so small near the set of minimizers.

\section{Main Contributions}\label{section results}

\subsection{Optimization in continuous time}

In this section, we study a continuous time version of gradient descent with (heavy ball or Nesterov) momentum derived by \citet{su2016differential}. Namely, we study solutions of the heavy ball ODE
\begin{equation}\label{eq heavy ball}
\ddot x + \gamma \,\dot x = - \nabla f(x), \qquad x(0) = x_0, \qquad \dot x(0) = 0.
\end{equation}

This is a popular model for the study of accelerated methods in optimization which avoids some of the technicalities of discrete time stepping algorithms while relying on the same core geometric concepts. Our main result is the following.

\begin{restatable}{theorem}{continuoustime}[Continuous time convergence guarantee]\label{theorem continuous strongly convex}
    Assume that $f$ satisfies the assumptions of Section \ref{section assumptions}, $\gamma = 2\sqrt\mu$ and that $x_0$ satisfies $f(x_0)<\alpha$ where $\alpha$ is as in Section \ref{section assumptions}. Then there exists a unique solution $x(t)=x_t$ of \eqref{eq heavy ball} and $f(x_t)<\alpha$ for all $t>0$. Furthermore, $x$ is $C^2$-smooth and
    \[
    f(x_t) - \inf_{z\in\R^d}f(z) \leq e^{-\sqrt\mu\,t} \left(f(x_0) - \inf f + \frac\mu2\,\dist(x_0, \mathcal M)^2\right).
    \]
\end{restatable}

The proof can be found in Appendix \ref{appendix acceleration continuous}. There, we prove that the Lyapunov function
\[
\L(t) = f(x_t) - f(\pi(x_t)) + \frac12\left\|\dot x_t + \sqrt\mu\big(x_t- \pi(x_t)\big)\right\|^2
\]
satisfies $\L(t) \leq e^{-\sqrt\mu t}\L(0)$. The function $\L$ can be considered as a modified energy with potential energy $f - \inf f + \frac\mu2\|x-\pi(x)\|^2$ and kinetic energy $\frac12\|\dot x\|^2$, but with a more complex quadratic term in which velocity and position interact. Going back to \citep{su2016differential}, this strategy of proof is common in convex optimization. The main difficulty in our more general geometric setting is that the time derivative of $\pi(x_t)$ (i.e.\ tangential velocity) is not zero in general, unless the minimizer is unique in $\mathcal U_\alpha.$ Thus, we have to additionally control for the interaction of the tangential velocity with the other terms in the Lyapunov sequence. One key observation that helps us is that the connecting line $x_t-\pi(x_t)$ is orthogonal to the velocity of $\pi(x_t)$. The other technical tool is the following lemma.

\begin{restatable}{lemma}{parallelmove}\label{lemma auxiliary parallel movement}
    Let $\mathcal M$ be a $C^2$-submanifold of $\R^d$ and $\mathcal U$ an open set containing $\mathcal M$ such that there exists a unique closest point projection $\pi: \mathcal U\to \mathcal M$. Let $x:(-\eps,\eps)\to \mathcal U$ be a $C^1$-curve and $z(t):= \pi\circ x(t)$. Then $\langle \dot x, \dot z\rangle \geq 0$ on $(-\eps, \eps)$.
\end{restatable}

Geometrically, Lemma \ref{lemma auxiliary parallel movement} states that the closest point projection $z$ of a point $x$ does not move in the opposite direction when we move $x$. Despite its geometric simplicity, Lemma \ref{lemma auxiliary parallel movement} is non-trivial and a crucial ingredient in our proofs. Its proof is given in Appendix \ref{appendix differential geometry}.

\citet{venturi2019spurious} show that loss functions in overparametrized deep learning do not have strict local minimizers outside the set of global minimizers (under suitable assumptions). Under a quantitative version of this geometric assumption, we obtain a global version of Theorem \ref{theorem continuous strongly convex}, albeit with less precise quantitative guarantees.

\begin{restatable}{theorem}{globalconvergence}\label{theorem global convergence}
    In addition to the assumptions of Section \ref{section assumptions}, assume that 
    \begin{enumerate}
        \item for every $R>\inf f$, there exists $L_R>0$ such that $\nabla f$ is Lipschitz-continuous with Lipschitz constant $L_R$ on $\mathcal U_R = \{x : f(x) <R\}$ and
        \item there exists a value $\delta>0$ such that $\|\nabla f(x)\|^2 < \delta$ implies that $f(x) -\inf f < \frac\alpha2$.
    \end{enumerate}
    Then, for any $x_0\in \R^d$, there exists $T\geq 0$ such that $x(t) \in \mathcal U_\alpha$ for all $t\geq T$ and such that $f(x(t)) - \inf f \leq \big(3\alpha/2 + \dist(x_T, \mathcal M)^2\big)e^{\sqrt\mu(T-t)}$ for all $t>T$.
\end{restatable}

The rate of decay here is $e^{-\sqrt\mu\,t}$, as compared to $e^{-\mu t}$ for gradient flow. We see more clearly why we speak of acceleration in discrete time (see the discussion below Theorem \ref{theorem discrete}).

\begin{remark}
    We conjecture that Theorem \ref{theorem continuous strongly convex} remains valid if the set $\mathcal M$ is convex rather than a smooth manifold. Then $\pi$ is defined globally, but generally only Lipschitz-continuous and not smooth.
\end{remark}

\subsection{Deterministic optimization in discrete time}

We can equivalently write the heavy ball ODE \eqref{eq heavy ball} as a system of two first order ODEs:
\begin{equation}\label{eq heavy ball system}
\begin{pde}
    \dot x &= v\\
    \dot v &= -2\sqrt\mu\,v-\nabla f(x),
\end{pde}\qquad x(0) = x_0, \qquad v(0) = 0.
\end{equation}
We choose a time-stepping scheme
\begin{equation}\label{eq nesterov}
x'_n = x_n + \sqrt\eta v_n, \qquad g_n= \nabla f(x'_n),\qquad x_{n+1} = x'_n - \eta g_n, \qquad v_{n+1} = \rho \big(v_n- \sqrt\eta g_n\big)
\end{equation}
for $\rho = \frac{1-\sqrt{\mu\eta}}{1+\sqrt{\mu\eta}} = 1-2\sqrt{\mu\eta} + O(\eta)$. In particular, we have
\[
x_{n+1} = x_n + \sqrt\eta\,v_n + O(\eta), \qquad v_{n+1} = v_n-\sqrt\eta\big(2\sqrt\mu\,v_n +\nabla f(x_n)\big) + O(\eta),
\]
i.e.\ the scheme is a time discretization of the heavy ball system \eqref{eq heavy ball system} with time-step size $\sqrt\eta$. The square root is chosen for consistency with the literature. This scheme is a geometrically intuitive reparametrization of Nesterov's accelerated gradient descent algorithm, except for the fact that Nesterov's scheme typically begins with the gradient descent step rather than the momentum step \citep[see][Appendix B, for a proof of equivalence]{gupta2023achieving}.

In discrete time, we need to make additional quantitative regularity assumptions on both $f$ and the projection map $\pi$ in order to ensure that the time step size is sufficiently small to recover the continuous time behavior. Note that the Hessian of the objective function $f$ is positive semi-definite on the set of global minimizers $\mathcal M$, i.e.\ it stands to reason that any negative eigenvalues of $D^2f$ should be small, at least close to $\mathcal M$. We note the following.

\begin{restatable}{lemma}{smallnegeigen}\label{lemma small negative eigenvalues}
Assume that $D^2f(x)\geq -\eps$ in a ball $B_r(x_0)$. Then
\[
\langle \nabla f(x), x-z\rangle \geq f(x) - f(z) - \frac\eps2\,\|x-z\|^2\qquad \forall\ x,z\in B_r(x_0).
\]
\end{restatable}
This lemma informs our geometric assumptions on the objective function $f$. For the closest point projection $\pi$, we assume that $\pi(x) =  P x + x^*$ for some (linear) orthogonal projection $ P$ onto a subspace $V\subseteq\R^d$ and a fixed vector $x^*$ which is the unique element with the smallest Euclidean norm in the affine space $\tilde V = x^*+V$. The assumption that the derivative $D\pi\equiv  P$ is constant is a (very restrictive) geometric linearization, and relaxing it is an important subject of future research. Still, it applies to many functions which are not convex, such as those with a unique minimizer (i.e.\ $ P\equiv 0$) or those in Example \ref{example product straight}. With an eye towards stochastic optimization, we opt for the simpler global geometric assumptions, see also Remark \ref{remark global projections}.

\begin{restatable}{theorem}{discrete}\label{theorem discrete}
Assume that $f$ is $L$-smooth and the sequences $x_n, x_n', v_n$ are generated according to the Nesterov scheme \eqref{eq nesterov} with parameters $\eta \leq 1/ L$ and $\rho = (1-\sqrt{\mu\eta})/(1+\sqrt{\mu\eta})$. Assume further that
there exists an affine linear projection map $\pi(x) =  P x+x^*$ such that
\begin{equation}\label{eq assumption convexity wrt minimizer}
\langle \nabla f(x), x-\pi(x)\rangle \geq f(x) -f(\pi(x)) + \frac\mu2 \,\|x-\pi(x)\|^2.
\end{equation}
Finally, assume that for arbitrary $x, v\in\R^d$ we have
\begin{equation}\label{eq assumption mild non-convexity}
\langle \nabla f(x+ v), v\rangle \geq f(x+ v) - f(x) - \frac\eps2\,\| v\|^2%
\end{equation}
with some $\eps \leq \sqrt{\mu/\eta}$. Then
\begin{align*}
    f(x_n) - \inf f \leq (1-\sqrt{\mu\eta})^n \left[ f(x_0) - \inf f + \frac\mu2 \|x_0 - \pi(x_0)\|^2 \right].
\end{align*}
\end{restatable}

The proof, given in Appendix \ref{appendix acceleration discrete}, builds on similar ideas as Theorem \ref{theorem continuous strongly convex}, with additional complications introduced by the discrete-time setting. We have to modify the usual Lyapunov sequence used for strongly convex functions. We treat the tangential and normal components of the velocity (i.e.\ $ P v_n$ and $ P^\bot v_n = (I- P)v_n$) separately, and carefully choose coefficients to ensure that the following Lyapunov sequence decays at each step,
\begin{align*}
    \L_n = f(x_n) - \inf f + \frac12\|\  P^\bot v_n + \sqrt\mu (x_n' -\pi(x_n')) \|^2 + \frac{(1+\sqrt{\mu\eta})^2}{2(1-\sqrt{\mu\eta})}\| P v_n\|^2.
\end{align*}
In Theorem \ref{theorem discrete}, we see more clearly than in Theorem \ref{theorem continuous strongly convex} why we talk of acceleration: While gradient descent would achieve a decay rate of $(1-\mu/L)^n$ with the commonly proposed step size $\eta=1/L$ in discrete time based on our assumptions, Nesterov's method achieves decay like $(1-\sqrt{\mu/L})^n$ with $\eta = 1/L$. Since $\mu/L\leq 1$, Nesterov's method converges much faster than gradient descent.

\begin{remark}\label{remark magnitude negative eigenvalues}
Note that the negative eigenvalues of the Hessian may be as large as $\sqrt{\mu/\eta}$ in Theorem \ref{theorem discrete}. If $\eta$ is chosen as large as $1/L$, this is a real restriction of the eigenvalues to the range $[-\sqrt{\mu L}, L]$. However, since the Hessian eigenvalues of an $L$-smooth function are in $[-L, L]$ a priori, there is no additional restriction if $\sqrt{\mu/\eta}\geq L$, i.e.\ if $\eta < \frac{\mu}{L^2}$. This corresponds to the continuous time guarantee of Theorem \ref{theorem continuous strongly convex}, which does not depend on the magnitude of the negative eigenvalues. 

However, if the eigenvalues of $D^2f$ are as negative as $-L$, the step size $\eta = \mu/L^2$ is so small that it does not improve upon gradient descent with step size $\eta=1/L$ since $\sqrt{\mu\eta}= \mu/L$ in this case. Thus, if $f$ is too far from being convex, acceleration may not be achievable in discrete time. Notably, especially close to the set of global minimizers, we can picture $\eps$ as small compared to $L$.

\end{remark}

\subsection{Stochastic optimization in discrete time}

In typical applications in deep learning, the gradient $\nabla f$ of the objective function/loss function $f$ is prohibitively expensive to evaluate, but we have access to stochastic estimates of the true gradient. In this section, in addition to the assumptions of Theorem \ref{theorem discrete}, we assume that we are given a probability space $(\Omega, \mathcal A, \mathbb Q)$ and a measurable function $g:\R^d\times \Omega \to\R^d$ such that
    \begin{equation}\label{eq stochastic gradient}
    \E_{\omega\sim\mathbb Q}\big[g(x,\omega)\big]= \nabla f(x), \qquad \E_{\omega\sim\mathbb Q}\big[\|g(x,\omega)\|^2\big] < +\infty\qquad\forall\ x\in \R^d.
    \end{equation}
For quantitative statements, a more precise assumption on the variance of the gradient estimates must be made. We make the modelling assumption
\begin{equation}\label{eq variance bound}
    \E_{\omega\sim\mathbb Q}\big[\|g(x,\omega) - \nabla f(x)\|^2\big]
    \leq \sigma_a^2 + \sigma_m^2\|\nabla f(x)\|^2\qquad\forall\ x\in\R^d.
\end{equation}
We call $\sigma_a$ the additive standard deviation and $\sigma_m$ the multiplicative standard deviation since they resemble the prototypical example $g = (1+\sigma_m N_1)\nabla f + \sigma_aN_2$ where $N_1, N_2$ are random variables with mean zero and variance one. The case of purely additive noise (i.e.\ $\sigma_m=0$) is classical and hails back to the seminal article of \cite{robbins1951stochastic}. The case of purely multiplicative noise is much closer to reality in overparametrized learning: If all data points can be fit exactly, there is no noise when estimating the gradient of the empirical risk/training loss on the set of global minimizers. It has received significant attention more recently by \cite{liu2018mass, bassily2018exponential, vaswani2019fast, even2021continuized, wojtowytsch2023stochastic, gupta2023achieving} and others.

Let us consider the purely additive case first. We follow the scheme \eqref{eq nesterov}, but we replace the deterministic gradient $\nabla f(x_n')$ by $g_n= g(x_n', \omega_n)$ where $\omega_0, \omega_1, \dots$ are drawn from $\Omega$ independently of each other and the initial condition $x_0$ with law $\mathbb Q$. This framework allows e.g.\ for minibatch sampling (but assumes that all batches are drawn independently of each other from the dataset).

\begin{restatable}{theorem}{additive}[Acceleration with additive noise]\label{theorem stochastic additive}
Assume that $f, P$ are as in Theorem \ref{theorem discrete} and that the $g$ satisfies \eqref{eq stochastic gradient} and \eqref{eq variance bound} with $\sigma_m=0$. Assume that the sequences $x_n, x_n', v_n$ are generated by the scheme \eqref{eq nesterov} for parameters $\eta \leq 1/L$ and $\rho = \frac{1-\sqrt{\mu\eta}}{1+\sqrt{\mu\eta}}$, but with the stochastic gradient estimates $g(x_n', \omega_n)$ with independently identically distributed $\omega_n$ in place of $\nabla f(x_n')$. Then
\begin{align*}
\mathbb{E}[f(x_n) - \inf f] &\leq (1-\sqrt{\mu\eta})^n \left[f(x_0) - \inf f + \frac{\mu}{2}\|x_0 - \pi(x_0)\|^2\right] \
+ \frac{\sigma_a^2\sqrt\eta}{\sqrt{\mu}}.
\end{align*}
\end{restatable}

Thus Nesterov's method reduces $f(x_n)$ below a `noise level' proportional to $\sigma_a^2\sqrt{\eta/\mu}$ at a linear rate $(1-\sqrt{\mu\eta})^n$ in the presence of additive noise. 
The analogous bound for stochastic gradient descent was obtained (in the more general setting of PL functions) in \citep[Theorem 4]{karimi2016linear} as
\[
\mathbb{E}[f(x_n) - \inf f] \leq (1-\mu\eta)^k \,\E\big[f(x_0) - \inf f\big] + \frac{L\sigma_a^2\eta}{2\mu}.
\]
With the largest admissible learning rate $\eta =1/L$, the noise level for GD is $\sigma_a^2/2\mu$ compared to the usually much lower value $\sigma_a^2/ \sqrt{L\mu}$ for Nesterov's method. Keeping a memory of previous gradient estimates facilitates `averaging out' the random noise.

With a fixed positive learning rate $\eta$, generally $f(x_n)\not\to 0$ unless $\sigma_a = 0$. We therefore consider a sequence of decreasing step sizes.

\begin{restatable}{theorem}{decreasingstep}[Additive noise and decreasing step size]\label{theorem decreasing learning rate}
    Assume that $f, g$ are as in Theorem \ref{theorem stochastic additive} and that the sequences $x_n, x_n', \rho_n$ are generated by the scheme
    \[
    x_n' = x_n + \sqrt{\eta_{n-1}} v_n, \qquad x_{n+1} = x_n' - \eta_n g_n, \qquad v_{n+1} = \rho_n(v_n - \sqrt{\eta_n} g_n)
    \]
    for parameters $\eta_n = \frac{\mu}{(n+\sqrt{L\mu}+1)^2}$, $\rho_n = \frac{1-\sqrt{\mu \eta_n}}{1+\sqrt{\mu\eta_n}}$.
    If $\eps\leq \sqrt{\mu/\eta_0} =\mu+\sqrt{L\mu}$, then
    \[
    \E\big[f(x_n) - \inf f\big] 
    \leq \frac{\sqrt{\frac L\mu}\,\E\left[f(x_0) - \inf f + \frac12\|x_0-\pi(x_0)\|^2\right] + \frac{\sigma_a^2}{\mu}\,\log\big(1+n\sqrt{\mu/L}\big)}{n+\sqrt{L/\mu}}.
    \]
\end{restatable}

Note that the `physical' step size $\sqrt{\eta_n}$ decays as $1/n$ and thus satisfies the (non-)summability conditions of \citet{robbins1951stochastic}. We can allow for larger $\eps$ by choosing $\eta_n = \mu/(n+n_0)$ with $n_0 > \sqrt{L/\mu}-1$ for a smaller initial step size.

If $\sigma_m>0$, \citet{liu2018mass} and \citet{gupta2023achieving} show that Nesterov's scheme no longer achieves acceleration. For general noise, we therefore consider a modified Nesterov scheme:
\begin{equation}\label{eq agnes}
x_n' = x_n + \sqrt{\alpha}v_n, \qquad x_{n+1} = x'_n - \eta g_n, \qquad v_{n+1} = \rho_n \big(v_n - \sqrt{\alpha}g_n)
\end{equation}
where again $g_n = g(x_n', \omega_n)$. The scheme \eqref{eq agnes} was introduced as the Accelerated Gradient method with Noisy EStimators method (AGNES) by \cite{gupta2023achieving} in convex and strongly convex optimization with purely multiplicative noise. Compared to Nesterov's algorithm, AGNES has an additional parameter $\alpha$ which is required to adapt to the multiplicative variance, at least if $\sigma_m\geq 1$. Here, we generalize the work of \cite{gupta2023achieving} by both allowing noise with general scaling for $\sigma_a, \sigma_m>0$ and relaxing the convexity assumption on $f$.

\begin{restatable}{theorem}{agnes}[Additive and multiplicative noise]\label{theorem agnes}
Assume that $f,  P, x^*$ are as in Theorem \ref{theorem discrete} and that $g$ is a family of gradient estimators such that \eqref{eq stochastic gradient} and \eqref{eq variance bound} hold for some $\sigma_a, \sigma_m\geq 0$.
Assume that the sequences $x_n, x_n', v_n$ are generated by the AGNES scheme \eqref{eq agnes} with parameters
\[
0<\eta \leq \frac1{L(1+\sigma_m^2)}, \qquad \rho =  \frac{1-\sqrt{\frac {\mu\eta}{1+\sigma_m^2}}}{1+\sqrt{\frac{\mu\eta}{1+\sigma_m^2}}}, \qquad\alpha =  \frac{1-\sqrt{\mu(1+\sigma_m^2)\eta}}{1-\sqrt{\mu(1+\sigma_m^2)\eta}+\sigma_m^2}\,\eta.
\]
Then, if $\eps <\sqrt{\mu(1+\sigma_m^2)/\eta}$, we have
\begin{align*}
\E\big[f(x_n) - \inf f\big] %
    &\leq\left(1-\sqrt{\frac{\mu\eta}{1+\sigma_m^2}}\right)^n \E\left[f(x_0)-\inf f + \frac\mu2 \norm{x_0-\pi(x_0)}^2 \right]
    + \frac{\sigma_a^2\sqrt\eta}{\sqrt{\mu(1+\sigma_m^2)}}.
\end{align*}
\end{restatable}

The proof of Theorem \ref{theorem agnes} is given in Appendix \ref{appendix agnes}. While at a glance it appears that the multiplicative noise is helping us reduce the additive error term, this is merely a consequence of the small learning rate which is forced upon us. The condition on the negative eigenvalues is relaxed to $\eps \leq \sqrt{\mu L}(1+\sigma_m^2)$ with the largest admissible step size $\eta = 1/(L(1+\sigma_m^2))$ as the issues stemming from tangential drift pale in comparison to those stemming from stochastic gradient estimates. For comparison, if we chose the same $\eta$ in Theorem \ref{theorem discrete}, we could only allow for $\eps \leq \sqrt{\mu L}\,\sqrt{1+\sigma_m^2} < \sqrt{\mu L}(1+\sigma_m^2)$, but we would obtain a rate of convergence of $1- \sqrt{\mu/{L(1+\sigma_m^2)}}$. In the stochastic case, we only achieve $1-\sqrt{\mu/L}/(1+\sigma_m^2)$. Thus the limiting factor is the stochastic noise, not the geometry of $f$.

\begin{remark}\label{remark global projections}
    We opted for a {\em linear} closest point projection map to facilitate proofs. In the non-linear case, closest point projections cannot be defined globally: \citet{jessen1940saetninger,busemann1947note, phelps1957convex} show that if $K$ is a subset of $\R^d$ such that for every $x\in\R^d$ there exists a unique closest point in $K$, then $K$ is closed and convex. If $K$ is both a $k$-dimensional submanifold of $\R^d$ and a convex set, then $K$ is an affine $k$-dimensional subspace of $\R^d$, i.e.\ our assumptions are the most general when assuming that a unique closest point projection onto a submanifold is defined globally.
    
    Thus, if $\mathcal M$ is not an affine space, we can only assume that $\pi$ is `good' in a neighborhood of $\mathcal M$. In stochastic optimization, where we can randomly `jump' out of the good neighborhood, this leads to serious technical challenges. Guarantees on `remaining local' with high probability have recently been derived for SGD with additive noise and decaying learning rates by \citet{mertikopoulos2020almost} and with multiplicative noise by \citet{wojtowytsch2023stochastic}. To avoid obscuring the new geometric constructions, we opted to forgo this highly technical setting here and prioritize the extension of Theorems \ref{theorem continuous strongly convex} and \ref{theorem discrete} towards stochastic optimization.
\end{remark}

\section{Conclusion}\label{section conclusion}

We have proved that first order momentum-based methods accelerate convergence in a more general setting than convex optimization with many geometric features motivated by loss landscapes encountered in deep learning. The models we studied include the heavy ball ODE and deterministic and stochastic optimization schemes in discrete time under various noise assumptions.

The most cumbersome limitation of our convergence guarantees is the assumption that the closest point projection onto the set of minimizers is affine linear in the discrete time setting, i.e.\ the derivative $D\pi$ is constant in space. Future work will focus on relaxing this assumption.

\newpage

\section*{Acknowledgements}
The authors gratefully acknowledge the financial support of the NSF, grant DMS 2424801. The authors are grateful to Quentin Merigot, who pointed out a simpler proof of Lemma \ref{lemma auxiliary parallel movement} which requires less regularity than the authors' original approach.

\bibliographystyle{iclr2025_conference}
\bibliography{bibliography.bib}

\begin{thebibliography}{52}
\providecommand{\natexlab}[1]{#1}
\providecommand{\url}[1]{\texttt{#1}}
\expandafter\ifx\csname urlstyle\endcsname\relax
  \providecommand{\doi}[1]{doi: #1}\else
  \providecommand{\doi}{doi: \begingroup \urlstyle{rm}\Url}\fi

\bibitem[Alain et~al.(2018)Alain, Roux, and Manzagol]{alain2018negative}
Guillaume Alain, Nicolas~Le Roux, and Pierre-Antoine Manzagol.
\newblock Negative eigenvalues of the hessian in deep neural networks, 2018.
\newblock URL \url{https://openreview.net/forum?id=S1iiddyDG}.

\bibitem[Apidopoulos et~al.(2022)Apidopoulos, Ginatta, and
  Villa]{apidopoulos2022convergence}
Vassilis Apidopoulos, Nicol{\`o} Ginatta, and Silvia Villa.
\newblock Convergence rates for the heavy-ball continuous dynamics for
  non-convex optimization, under polyak--{\l}ojasiewicz condition.
\newblock \emph{Journal of Global Optimization}, 84\penalty0 (3):\penalty0
  563--589, 2022.

\bibitem[Aujol et~al.(2022)Aujol, Dossal, and
  Rondepierre]{aujol2022convergence}
J-F Aujol, Ch~Dossal, and Aude Rondepierre.
\newblock Convergence rates of the heavy ball method for quasi-strongly convex
  optimization.
\newblock \emph{SIAM Journal on Optimization}, 32\penalty0 (3):\penalty0
  1817--1842, 2022.

\bibitem[Aujol et~al.(2024)Aujol, Dossal, Labarrière, and
  Rondepierre]{aujol2024heavyballmomentumnonstrongly}
Jean-François Aujol, Charles Dossal, Hippolyte Labarrière, and Aude
  Rondepierre.
\newblock Heavy ball momentum for non-strongly convex optimization, 2024.
\newblock URL \url{https://arxiv.org/abs/2403.06930}.

\bibitem[Bassily et~al.(2018)Bassily, Belkin, and Ma]{bassily2018exponential}
Raef Bassily, Mikhail Belkin, and Siyuan Ma.
\newblock On exponential convergence of sgd in non-convex over-parametrized
  learning.
\newblock \emph{arXiv preprint arXiv:1811.02564}, 2018.

\bibitem[Beck \& Teboulle(2009)Beck and Teboulle]{beck2009fast}
Amir Beck and Marc Teboulle.
\newblock A fast iterative shrinkage-thresholding algorithm for linear inverse
  problems.
\newblock \emph{SIAM journal on imaging sciences}, 2\penalty0 (1):\penalty0
  183--202, 2009.

\bibitem[Busemann(1947)]{busemann1947note}
Herbert Busemann.
\newblock Note on a theorem on convex sets.
\newblock \emph{Matematisk Tidsskrift. B}, pp.\  32--34, 1947.

\bibitem[Cauchy et~al.(1847)]{cauchy1847methode}
Augustin Cauchy et~al.
\newblock M{\'e}thode g{\'e}n{\'e}rale pour la r{\'e}solution des systemes
  d’{\'e}quations simultan{\'e}es.
\newblock \emph{Comp. Rend. Sci. Paris}, 25\penalty0 (1847):\penalty0 536--538,
  1847.

\bibitem[Cooper(2021)]{cooper2021global}
Yaim Cooper.
\newblock Global minima of overparameterized neural networks.
\newblock \emph{SIAM Journal on Mathematics of Data Science}, 3\penalty0
  (2):\penalty0 676--691, 2021.

\bibitem[Damian et~al.(2021)Damian, Ma, and Lee]{damian2021label}
Alex Damian, Tengyu Ma, and Jason~D Lee.
\newblock Label noise sgd provably prefers flat global minimizers.
\newblock \emph{Advances in Neural Information Processing Systems},
  34:\penalty0 27449--27461, 2021.

\bibitem[Du et~al.(2017)Du, Jin, Lee, Jordan, Singh, and
  Poczos]{du2017gradient}
Simon~S Du, Chi Jin, Jason~D Lee, Michael~I Jordan, Aarti Singh, and Barnabas
  Poczos.
\newblock Gradient descent can take exponential time to escape saddle points.
\newblock \emph{Advances in neural information processing systems}, 30, 2017.

\bibitem[E et~al.(2019)E, Ma, and Wu]{weinan2019comparative}
Weinan E, Chao Ma, and Lei Wu.
\newblock A comparative analysis of optimization and generalization properties
  of two-layer neural network and random feature models under gradient descent
  dynamics.
\newblock \emph{Sci. China Math}, 2019.

\bibitem[Even et~al.(2021)Even, Berthier, Bach, Flammarion, Gaillard, Hendrikx,
  Massouli{\'e}, and Taylor]{even2021continuized}
Mathieu Even, Rapha{\"e}l Berthier, Francis Bach, Nicolas Flammarion, Pierre
  Gaillard, Hadrien Hendrikx, Laurent Massouli{\'e}, and Adrien Taylor.
\newblock A continuized view on nesterov acceleration for stochastic gradient
  descent and randomized gossip.
\newblock \emph{arXiv preprint arXiv:2106.07644}, 2021.

\bibitem[Fu et~al.(2023)Fu, Xu, and Wilson]{fu2023accelerated}
Qiang Fu, Dongchu Xu, and Ashia~Camage Wilson.
\newblock Accelerated stochastic optimization methods under quasar-convexity.
\newblock In \emph{International Conference on Machine Learning}, pp.\
  10431--10460. PMLR, 2023.

\bibitem[Ghadimi \& Lan(2012)Ghadimi and Lan]{ghadimi2012optimal}
Saeed Ghadimi and Guanghui Lan.
\newblock Optimal stochastic approximation algorithms for strongly convex
  stochastic composite optimization i: A generic algorithmic framework.
\newblock \emph{SIAM Journal on Optimization}, 22\penalty0 (4):\penalty0
  1469--1492, 2012.

\bibitem[Ghadimi \& Lan(2013)Ghadimi and Lan]{ghadimi2013optimal}
Saeed Ghadimi and Guanghui Lan.
\newblock Optimal stochastic approximation algorithms for strongly convex
  stochastic composite optimization, ii: shrinking procedures and optimal
  algorithms.
\newblock \emph{SIAM Journal on Optimization}, 23\penalty0 (4):\penalty0
  2061--2089, 2013.

\bibitem[Goujaud et~al.(2023)Goujaud, Taylor, and
  Dieuleveut]{goujaud2023provable}
Baptiste Goujaud, Adrien Taylor, and Aymeric Dieuleveut.
\newblock Provable non-accelerations of the heavy-ball method.
\newblock \emph{arXiv preprint arXiv:2307.11291}, 2023.

\bibitem[Guminov et~al.(2023)Guminov, Gasnikov, and
  Kuruzov]{guminov2023accelerated}
Sergey Guminov, Alexander Gasnikov, and Ilya Kuruzov.
\newblock Accelerated methods for weakly-quasi-convex optimization problems.
\newblock \emph{Computational Management Science}, 20\penalty0 (1):\penalty0
  36, 2023.

\bibitem[Gupta et~al.(2024)Gupta, Siegel, and Wojtowytsch]{gupta2023achieving}
Kanan Gupta, Jonathan~W. Siegel, and Stephan Wojtowytsch.
\newblock Nesterov acceleration despite very noisy gradients.
\newblock In \emph{Advances in Neural Information Processing Systems},
  volume~37, pp.\  20694--20744, 2024.
\newblock URL
  \url{https://proceedings.neurips.cc/paper_files/paper/2024/file/24d2dd6dc9b79116f8ebc852ddb9dc94-Paper-Conference.pdf}.

\bibitem[Hermant et~al.(2024)Hermant, Aujol, Dossal, and
  Rondepierre]{hermant2024studybehaviournesterovaccelerated}
J~Hermant, J.~F Aujol, C~Dossal, and A~Rondepierre.
\newblock Study of the behaviour of nesterov accelerated gradient in a non
  convex setting: the strongly quasar convex case, 2024.
\newblock URL \url{https://arxiv.org/abs/2405.19809}.

\bibitem[Hestenes \& Stiefel(1952)Hestenes and Stiefel]{hestenes1952methods}
Magnus~Rudolph Hestenes and Eduard Stiefel.
\newblock \emph{Methods of conjugate gradients for solving linear systems}.
\newblock NBS Washington, DC, 1952.

\bibitem[Hinder et~al.(2020)Hinder, Sidford, and Sohoni]{hinder2020near}
Oliver Hinder, Aaron Sidford, and Nimit Sohoni.
\newblock Near-optimal methods for minimizing star-convex functions and beyond.
\newblock In \emph{Conference on learning theory}, pp.\  1894--1938. PMLR,
  2020.

\bibitem[Jacot et~al.(2018)Jacot, Gabriel, and Hongler]{jacot2018neural}
Arthur Jacot, Franck Gabriel, and Cl{\'e}ment Hongler.
\newblock Neural tangent kernel: Convergence and generalization in neural
  networks.
\newblock \emph{Advances in neural information processing systems}, 31, 2018.

\bibitem[Jain et~al.(2019)Jain, Nagaraj, and Netrapalli]{jain2019making}
Prateek Jain, Dheeraj Nagaraj, and Praneeth Netrapalli.
\newblock Making the last iterate of sgd information theoretically optimal.
\newblock In \emph{Conference on Learning Theory}, pp.\  1752--1755. PMLR,
  2019.

\bibitem[Jessen(1940)]{jessen1940saetninger}
B{\o}rge Jessen.
\newblock To saetninger om konvekse punktmaengder.
\newblock \emph{Matematisk tidsskrift. B}, pp.\  66--70, 1940.

\bibitem[Josz et~al.(2023)Josz, Lai, and Li]{josz2023convergence}
C{\'e}dric Josz, Lexiao Lai, and Xiaopeng Li.
\newblock Convergence of the momentum method for semialgebraic functions with
  locally lipschitz gradients.
\newblock \emph{SIAM Journal on Optimization}, 33\penalty0 (4):\penalty0
  3012--3037, 2023.

\bibitem[Karimi et~al.(2016)Karimi, Nutini, and Schmidt]{karimi2016linear}
Hamed Karimi, Julie Nutini, and Mark Schmidt.
\newblock Linear convergence of gradient and proximal-gradient methods under
  the {P}olyak-{L}ojasiewicz condition.
\newblock In \emph{Machine Learning and Knowledge Discovery in Databases:
  European Conference, ECML PKDD 2016, Riva del Garda, Italy, September 19-23,
  2016, Proceedings, Part I 16}, pp.\  795--811. Springer, 2016.

\bibitem[Kassing \& Weissmann(2024)Kassing and Weissmann]{kassing2024polyak}
Sebastian Kassing and Simon Weissmann.
\newblock Polyak's heavy ball method achieves accelerated local rate of
  convergence under {P}olyak-{L}ojasiewicz inequality.
\newblock \emph{arXiv preprint arXiv:2410.16849}, 2024.

\bibitem[Laborde \& Oberman(2020)Laborde and Oberman]{laborde2020lyapunov}
Maxime Laborde and Adam Oberman.
\newblock A lyapunov analysis for accelerated gradient methods: from
  deterministic to stochastic case.
\newblock In \emph{International Conference on Artificial Intelligence and
  Statistics}, pp.\  602--612. PMLR, 2020.

\bibitem[Lee et~al.(2019)Lee, Panageas, Piliouras, Simchowitz, Jordan, and
  Recht]{lee2019first}
Jason~D Lee, Ioannis Panageas, Georgios Piliouras, Max Simchowitz, Michael~I
  Jordan, and Benjamin Recht.
\newblock First-order methods almost always avoid strict saddle points.
\newblock \emph{Mathematical programming}, 176:\penalty0 311--337, 2019.

\bibitem[Lessard et~al.(2016)Lessard, Recht, and Packard]{lessard2016analysis}
Laurent Lessard, Benjamin Recht, and Andrew Packard.
\newblock Analysis and design of optimization algorithms via integral quadratic
  constraints.
\newblock \emph{SIAM Journal on Optimization}, 26\penalty0 (1):\penalty0
  57--95, 2016.

\bibitem[Li et~al.(2021)Li, Wang, and Arora]{li2021happens}
Zhiyuan Li, Tianhao Wang, and Sanjeev Arora.
\newblock What happens after sgd reaches zero loss?--a mathematical framework.
\newblock \emph{arXiv preprint arXiv:2110.06914}, 2021.

\bibitem[Liu \& Belkin(2018)Liu and Belkin]{liu2018mass}
Chaoyue Liu and Mikhail Belkin.
\newblock Mass: an accelerated stochastic method for over-parametrized
  learning.
\newblock \emph{arXiv preprint arXiv:1810.13395}, 2018.

\bibitem[Liu et~al.(2024)Liu, Drusvyatskiy, Belkin, Davis, and
  Ma]{liu2024aiming}
Chaoyue Liu, Dmitriy Drusvyatskiy, Misha Belkin, Damek Davis, and Yian Ma.
\newblock Aiming towards the minimizers: fast convergence of sgd for
  overparametrized problems.
\newblock \emph{Advances in neural information processing systems}, 36, 2024.

\bibitem[Liu et~al.(2022)Liu, Pan, and Tao]{liu2022provable}
Xin Liu, Zhisong Pan, and Wei Tao.
\newblock Provable convergence of nesterov’s accelerated gradient method for
  over-parameterized neural networks.
\newblock \emph{Knowledge-Based Systems}, 251:\penalty0 109277, 2022.

\bibitem[Mertikopoulos et~al.(2020)Mertikopoulos, Hallak, Kavis, and
  Cevher]{mertikopoulos2020almost}
Panayotis Mertikopoulos, Nadav Hallak, Ali Kavis, and Volkan Cevher.
\newblock On the almost sure convergence of stochastic gradient descent in
  non-convex problems.
\newblock \emph{Advances in Neural Information Processing Systems},
  33:\penalty0 1117--1128, 2020.

\bibitem[Necoara et~al.(2019)Necoara, Nesterov, and Glineur]{necoara2019linear}
Ion Necoara, Yu~Nesterov, and Francois Glineur.
\newblock Linear convergence of first order methods for non-strongly convex
  optimization.
\newblock \emph{Mathematical Programming}, 175:\penalty0 69--107, 2019.

\bibitem[Nemirovski et~al.(2009)Nemirovski, Juditsky, Lan, and
  Shapiro]{nemirovski2009robust}
Arkadi Nemirovski, Anatoli Juditsky, Guanghui Lan, and Alexander Shapiro.
\newblock Robust stochastic approximation approach to stochastic programming.
\newblock \emph{SIAM Journal on optimization}, 19\penalty0 (4):\penalty0
  1574--1609, 2009.

\bibitem[Nesterov(1983)]{nesterov_original}
Yurii Nesterov.
\newblock A method for solving the convex programming problem with convergence
  rate $o(1/k^2)$.
\newblock \emph{Dokl. Akad. Nauk SSSR}, 269:\penalty0 543--547, 1983.

\bibitem[O’Neill \& Wright(2019)O’Neill and Wright]{o2019behavior}
Michael O’Neill and Stephen~J Wright.
\newblock Behavior of accelerated gradient methods near critical points of
  nonconvex functions.
\newblock \emph{Mathematical Programming}, 176:\penalty0 403--427, 2019.

\bibitem[Phelps(1957)]{phelps1957convex}
RR~Phelps.
\newblock Convex sets and nearest points.
\newblock \emph{Proceedings of the American Mathematical Society}, 8\penalty0
  (4):\penalty0 790--797, 1957.

\bibitem[Rebjock \& Boumal(2023)Rebjock and Boumal]{rebjock2023fast}
Quentin Rebjock and Nicolas Boumal.
\newblock Fast convergence to non-isolated minima: four equivalent conditions
  for $c^2$-functions.
\newblock \emph{arXiv preprint arXiv:2303.00096}, 2023.

\bibitem[Robbins \& Monro(1951)Robbins and Monro]{robbins1951stochastic}
Herbert Robbins and Sutton Monro.
\newblock A stochastic approximation method.
\newblock \emph{The annals of mathematical statistics}, pp.\  400--407, 1951.

\bibitem[Sagun et~al.(2017)Sagun, Bottou, and LeCun]{sagun2017eigenvalues}
Levent Sagun, Leon Bottou, and Yann LeCun.
\newblock Eigenvalues of the hessian in deep learning: Singularity and beyond,
  2017.
\newblock URL \url{https://openreview.net/forum?id=B186cP9gx}.

\bibitem[Sagun et~al.(2018)Sagun, Evci, Guney, Dauphin, and
  Bottou]{sagun2018empirical}
Levent Sagun, Utku Evci, V.~Ugur Guney, Yann Dauphin, and Leon Bottou.
\newblock Empirical analysis of the hessian of over-parametrized neural
  networks, 2018.
\newblock URL \url{https://openreview.net/forum?id=rJrTwxbCb}.

\bibitem[Shamir \& Zhang(2013)Shamir and Zhang]{shamir2013stochastic}
Ohad Shamir and Tong Zhang.
\newblock Stochastic gradient descent for non-smooth optimization: Convergence
  results and optimal averaging schemes.
\newblock In \emph{International conference on machine learning}, pp.\  71--79.
  PMLR, 2013.

\bibitem[Su et~al.(2016)Su, Boyd, and Candes]{su2016differential}
Weijie Su, Stephen Boyd, and Emmanuel~J Candes.
\newblock A differential equation for modeling nesterov's accelerated gradient
  method: Theory and insights.
\newblock \emph{Journal of Machine Learning Research}, 17\penalty0
  (153):\penalty0 1--43, 2016.

\bibitem[Vaswani et~al.(2019)Vaswani, Bach, and Schmidt]{vaswani2019fast}
Sharan Vaswani, Francis Bach, and Mark Schmidt.
\newblock Fast and faster convergence of sgd for over-parameterized models and
  an accelerated perceptron.
\newblock In \emph{The 22nd international conference on artificial intelligence
  and statistics}, pp.\  1195--1204. PMLR, 2019.

\bibitem[Venturi et~al.(2019)Venturi, Bandeira, and Bruna]{venturi2019spurious}
Luca Venturi, Afonso~S Bandeira, and Joan Bruna.
\newblock Spurious valleys in one-hidden-layer neural network optimization
  landscapes.
\newblock \emph{Journal of Machine Learning Research}, 20\penalty0
  (133):\penalty0 1--34, 2019.

\bibitem[Wang \& Wibisono(2023)Wang and Wibisono]{wang2023continuized}
Jun-Kun Wang and Andre Wibisono.
\newblock Continuized acceleration for quasar convex functions in non-convex
  optimization.
\newblock \emph{arXiv preprint arXiv:2302.07851}, 2023.

\bibitem[Wojtowytsch(2023)]{wojtowytsch2023stochastic}
Stephan Wojtowytsch.
\newblock Stochastic gradient descent with noise of machine learning type part
  i: Discrete time analysis.
\newblock \emph{Journal of Nonlinear Science}, 33\penalty0 (3):\penalty0 45,
  2023.

\bibitem[Yue et~al.(2023)Yue, Fang, and Lin]{yue2023lower}
Pengyun Yue, Cong Fang, and Zhouchen Lin.
\newblock On the lower bound of minimizing {P}olyak-{\l}ojasiewicz functions.
\newblock In \emph{The Thirty Sixth Annual Conference on Learning Theory}, pp.\
   2948--2968. PMLR, 2023.

\end{thebibliography}

\newpage
\appendix

\section{Proofs of Lemmas \ref{lemma application} and \ref{lemma auxiliary parallel movement}: Geometry of the energy landscape}\label{appendix differential geometry}

We assume that the reader is familiar with basic concepts of differential geometry, such as submanifolds of Euclidean spaces and tangent spaces and with concepts of multi-variate analysis such as compactness and the inverse function theorem.

We first recall an important observation: Assume that $\mathcal M$ is a $C^1$-manifold. Fix a point $x\in\R^d$. Assume that the function $d:\mathcal M\to \R$ given by $d(z) = \|x-z\|$ has a local extremum at $z\in\mathcal M$. Then for every $C^1$-curve $\gamma:(-\eps, \eps)\to \mathcal M$ such that $\gamma(0) = z$, we have
\[
0 = \frac{d}{dt}\bigg|_{t=0} d(\gamma(t))^2 = 2\langle x- \gamma(0), \,\dot \gamma(0)\rangle = 2\langle x-z, \dot \gamma(0)\rangle
\]
or in other words: the connecting line $x-z$ is orthogonal to the tangent space $T_z\mathcal M$. This is in particular true if $z$ is the closest point in $\mathcal M$ to $x$.

\parallelmove*

\begin{proof}
    The closest point projection onto a $C^2$-manifold is $C^1$-smooth and satisfies 
    \begin{equation}\label{eq formula for pi}
    \pi(x) = x - d(x)\,\nabla d(x)
    \end{equation}
    where $d(x) = \dist(x, \mathcal M)$ is the distance function to the manifold, i.e.\ $\nabla d(x)$ gives the unit vector pointing directly towards the manifold. For details, see the proof of Lemma \ref{lemma application}.

    We can rewrite \eqref{eq formula for pi} as
    \begin{align*}
    \pi(x) &= \nabla \left(\frac{\|x\|^2}2 - \frac{d(x)^2}2\right) 
    = \nabla\left(\frac{\|x\|^2}2 - \frac{\min_{x\in\mathcal M}\|x-z\|^2}2\right)\\
    &= \nabla_x \max_{\mathcal M}\left(\frac{\|x\|^2}2 - \frac{\|x\|^2 - 2\langle x, z\rangle + \|z\|^2}2\right) = \nabla_x \max_{z\in\mathcal M}\left(\langle x, z\rangle - \frac{\|z\|^2}2\right).
    \end{align*}
    We are taking the (pointwise in $x$) maximum over a class of functions which are linear in $x$, i.e.\ we find that
    \[
    \xi(x) := \max_{z\in\mathcal M}\left(\langle x, z\rangle - \frac{\|z\|^2}2\right)
    \]
    is convex. In particular, the derivative matrix $D\pi = D^2\xi$ is symmetric and positive semi-definite and thus
    \[
    \left\langle \dot x, \frac{d}{dt}\pi\circ x\right\rangle =  \left\langle \dot x, \:D\pi(x)\, \dot x\right\rangle \geq 0.\qedhere
    \]
\end{proof}

\application*
\begin{proof}%
This result uses standard ideas from differential geometry. For the reader's convenience, we provide a full proof sketch. 

{\bf Step 1. Closest point projection.} Assume that $\mathcal M$ is a $C^m$-manifold for $m\geq 2$. Fix a point $z_0\in \mathcal M$, a radius $r>0$ and the neighbourhood $U= B_r(z_0)$. We assume that $r>0$ is so small that there exists a $C^m$-diffeomorphism $\phi:U\to V\subseteq \R^d$ such that $\phi(U\cap \mathcal M) = V\cap \{y : y_{k+1} = \dots = y_d = 0\}$. If $\mathcal M$ is a $C^m$-manifold, we can find a collection of $C^{m-1}$-smooth vector fields $A_{1}, \dots, A_d$ such that
\begin{enumerate}
\item $A_1(z), \dots, A_d(z)$ span the tangent space $T_z\mathcal M$ for all $z\in\mathcal M\cap U$ and
\item $\langle A_i(z), A_j(z)\rangle = \delta_{ij}$ for all $i,j=1,\dots, d$ and $z\in \mathcal M\cap U$.
\end{enumerate}
Such vector fields can be obtained for instance by applying the Gram-Schmidt algorithm to the columns of the derivative matrix $D(\phi^{-1})_{\phi(z)} = (D\phi_z)^{-1}$ of the inverse diffeomorphism. The algorithm returns an orthonormal basis since $\phi$ is a diffeomorphism, i.e.\ $D\phi$ has full rank. The first $k$ columns span the tangent space of $T_x\mathcal M$ since motion tangential to $\mathcal M$ in $U$ corresponds to motion where the last $d-k$ coordinates are kept zero in $V$, i.e.\ to the first $k$ columns of $D\phi^{-1}$.

We now introduce new coordinates: Denote by $\hat V\subseteq \R^k$ the set such that $\hat V \times \{0_{d-k}\} = V\cap \{y : y_{k+1} = \dots = y_d =0\}$  and
\[
\Psi : \hat V \times \R^{d-k} \to \R^d, \qquad \Psi(\hat y, s_{k+1}, \dots, s_d) = \phi^{-1} (\hat y, 0) + \sum_{i=k+1}^d s_i \,A_i\big(\phi^{-1}(\hat y, 0)\big).
\]
The map $\Psi$ is $C^{m-1}$-smooth since it is linear in $s$ and the least regular components, the vector fields $A_i$, are $C^{m-1}$-smooth in $y$. If $m\geq 2$, we trivially find that $\Psi$ is differentiable and $D\Psi_{(\hat y,0)} = (\partial_{\hat y_1}\phi, \dots, \partial_{\hat y_k}\phi, A_{k+1}, \dots, A_d)$ is invertible. Hence, the map $\Psi$ is a local diffeomorphism by the inverse function theorem. Notably, we see that 
\[
x- \Psi(\hat y, 0) \bot T_{\Psi(\hat y, 0)}\mathcal M \quad\LRa\quad x- \Psi(\hat y, 0) \in \mathrm{span} \left\{A_{k+1}(\Psi(\hat y, 0)), \dots, A_{d}(\Psi(\hat y, 0))\right\}
\]
and thus if and only if 
\[
x = \Psi(\hat y,0) + \sum_{i=k+1}^d s_iA_i \big(\Psi(\hat y, 0)\big) = \Psi(\hat y, s)
\]
for some $s\in \R^{d-k}$. In a neighbourhood of $z_0$ where $\Psi$ is a diffeomorphism, we set $\pi(\Psi(\hat y, s)) = \Psi(\hat y, 0)$, i.e.\ 
\[
\pi = \Psi \circ  P_{\R^k}\circ \Psi^{-1}, \qquad  P(y_1, \dots, y_d) = (y_1, \dots, y_k, 0, \dots, 0).
\]
The map is as smooth as $\Psi$, i.e.\ $C^{m-1}$-smooth (assuming that $m\geq 2$). The map $\pi$ defined in this way may not be the unique closest point projection on all of $U$ (e.g.\ when a point $z'\in\mathcal M\setminus U$ is closer), but it is guaranteed to be the unique closest point projection on a smaller subset $B_{r/2}(z^*)$ where the closest point on $\mathcal M$ is closer than the boundary of $U$.

Thus, for every point $z\in \mathcal M$, there exists a neighborhood $B_{r(z)}(z)$ for $r(z)>0$ in which a unique closest point projection is defined. Setting $U:= \bigcup_{z\in \mathcal M} B_{r(z)}(z)$, we find a neighborhood of the manifold $\mathcal M$ inside of which the closest point projection is defined. 

Assume that the radius $r(z)$ is chosen as the supremum of all admissible radii. Then the function $r(z)$ is strictly positive and Lipschitz-continuous with Lipschitz-constant $1$: $r(z') \geq r(z) - \|z-z'\|$ since $B_{r - \|z-z'\|}(z') \subseteq B_{r} (z)$. Exchanging the role of $z, z'$ shows that
\[
r(z') \geq r(z) - \|z-z'\|, \quad r(z) \geq r(z') - \|z-z'\|\qquad \Ra\quad |r(z) - r(z')| \leq \|z-z'\| \quad\forall\ z, z'\in \mathcal M.
\]
In particular, if $\mathcal M$ is compact, then as a continuous positive function, $r$ is uniformly positive and there exists a neighborhood $W_\delta := \{x\in \R^d : \dist(x, \mathcal M) < \delta\}$ on which the unique closest point projection is defined. Additionally, we find that
\[
\big\|\Psi(\hat y, s) - \pi\circ\Psi(\hat y, s)\big\|= \left\|\sum_{i=k+1}^ds_i A_i(\Psi(\hat y,0))\right\| = \|s\|
\]
since the vector fields $A_i$ are orthonormal. 

In Lemma \ref{lemma auxiliary parallel movement}, we require $\pi(x(t))$ to be $C^2$-smooth if $x(t)$ is $C^2$-smooth due to the technicalities of the proof. For this reason, we make the assumption that $\mathcal M$ is a $C^3$-manifold to ensure that $\pi$ is a $C^2$-map. For the required smoothness of solutions to the heavy ball ODE, it is sufficient to require that $f$ is $C^2$-smooth.

{\bf Step 2. The geometry of $f$.} Since $\mathcal M$ is the set of minimizers of $f$, the Hessian $D^2f(z)$ is positive semi-definite for every $z\in\mathcal M$. As $f$ is constant on $\mathcal M$, we for every curve $\gamma:(-\eps,\eps)\to \mathcal M$ we have
\[
 0 = \frac{d^2}{dt^2} f(\gamma(t)) = \nabla f(\gamma(t)), \,\gamma''(t)\rangle + \gamma'(t)^T\,D^2f(\gamma(t))\,\gamma'(t) = \gamma'(t)^T\,D^2f(\gamma(t))\,\gamma'(t)
\] 
because $\nabla f\equiv 0$ on the set $\mathcal M$ of minimizers of $f$, i.e.\ $v^TD^2f(z)v =0$ for all $z\in\mathcal M$ and $v\in T_z\mathcal M$. If we assume that $D^2f(z)$ has rank $d-k$ for all $z\in\mathcal M$, then necessarily $v^TD^2f(z)v >0$ for all $v$ which are orthogonal to $\mathcal M$ or equivalently
\[
v^TD^2f(z)v \geq \lambda(z)\big\| P_z^\bot v\big\|^2 \qquad\forall\ z\in \mathcal M, \:v\in\R^d
\]
where $ P^\bot_z$ denotes the orthogonal projection onto the orthogonal complement of the tangent space of $\mathcal M$ at $z$, i.e.\ 
\[
 P_z^\bot v = \sum_{i=k+1}^d \langle v, A_i(z)\rangle\,A_i(z).
\]
If $\mathcal M$ is compact, then the function $\lambda$ is bounded from below by some $\lambda_0>0$. Let $\eps = \lambda_0/2$. Using the uniform continuity of $\pi,  P$ and $D^2f$ on a compact set $\overline{W_\delta}$ and choosing $\delta>0$ suitably small, we find that
\[
v^TD^2f(x)v \geq \frac{\lambda_0}2\big\| P^\bot_{\pi(x)} v\|^2 -\eps \|v\|^2\geq \qquad \forall\ x\in \overline{W_\delta}, \:v\in S^{d-1}
\]
since the map from matrix to smallest eigenvalue is continuous on the space of symmetric matrices. In particular, for any fixed $(\hat y, s)\in \hat V\times S^{d-1}$ we see that the function 
\[
g:(-\delta, \delta)\to\R, \qquad g(t) =  f\big(\Psi(\hat y, ts)\big) 
\]
is $\lambda_0-\eps$ strongly convex. To see this, abbreviate $v:= \sum_{i=k+1}^d s^iA_i(\hat y)$ and compute 
\begin{align*}
g''(t) &= \frac{d^2}{dt^2} f\left(\Psi(\hat y,0) + tv\right) = v^TD^2f\left(\Psi(\hat y, ts)\right) v 
	\geq (\lambda_0-\eps)\left\|v\right\|^2 = \lambda_0-\eps
\end{align*}
since $v\in T_{\Psi(\hat y, 0)}\mathcal M^\bot$ and $\|v\| = \|s\|=1$. Hence
\begin{align*}
f\big(\Psi(\hat y, 0)\big) &= g(0) \geq g(t) - g'(t )t + \frac{\lambda_0-\eps}2 t^2\\
	&= f\big(\Psi(\hat y, ts)\big) - t\,\langle\nabla f(\Psi(\hat y, ts)), v\rangle + \frac{\lambda_0-\eps}2 \big\|tv\big\|^2\\
	&= f\big(\Psi(\hat y, ts)\big)  + \langle\nabla f(\Psi(\hat y, ts)), \Psi(\hat y, 0) - \Psi(\hat y, ts)\rangle + \frac{\lambda_0-\eps}2 \big\|\Psi(\hat y, 0) - \Psi(\hat y, ts)\big\|^2
\end{align*}
or in the original coordinates of $W_\delta\subseteq\R^d$:
\[
f(\pi(x)) \geq f(x) + \langle \nabla f(x), \pi(x) -x\rangle + \frac{\lambda_0-\eps}2\,\|x-\pi(x)\|^2.
\]
By the same argument with reversed roles for $x,\pi(x)$ we have
\[
f(x) \geq f(\pi(x)) + \frac {\lambda_0-\eps}2\|x-\pi(x)\|^2 = \inf f + \frac{\lambda_0-\eps}2\,\dist(x, \mathcal M)^2.
\]
In particular: $f(x) < \inf f + \alpha$ implies that
\[
\dist(x,\mathcal M) \leq \sqrt{\frac{2}{\lambda_0-\eps} \big(f(x) - \inf f\big)} < \sqrt{\frac{2\alpha}{\lambda_0-\eps}}.
\]
Choosing $\alpha$ small enough, we see that the open neighborhood 
\[
\mathcal U_\alpha := \{x : f(x) < \inf f+\alpha\}
\]
is a subset of $W_\delta$. Within this neighborhood, the unique closest point projection is therefore well-defined. This concludes the proof of the Lemma and shows that the Assumptions of Section \ref{section assumptions} are satisfied in this setting.
\end{proof}

\begin{remark}
Controlling the largest eigenvalue of the Hessian rather than the smallest, we see that there exist $0<\mu<L$ such that
\[
\frac\mu2\,\dist(x,\mathcal M)^2 \leq f(x) \leq \frac L2\,\dist(x, \mathcal M)^2
\]
in a neighborhood of $\mathcal U$. For this reason, we presented Example \ref{example distance squared} for context and intuition.
\end{remark}

\section{Proofs of Acceleration in Optimization}\label{appendix acceleration continuous}

\subsection{Proof of Theorems \ref{theorem continuous strongly convex} and \ref{theorem global convergence}: Optimization in continuous time}

We first prove the `local' convergence statement.

\continuoustime*
\begin{proof}%
{\bf Step 0: Existence and Uniqueness.} Note that a solution to the heavy ball ODE can be obtained as a solution to the ODE system
\[
\begin{pde}
\dot x &= v\\
\dot v &= -\gamma v - \nabla f(x).
\end{pde}
\]
If $\nabla f$ is locally Lipschitz-continuous (for instance if $f$ is $C^2$-smooth), a unique $C^1$-solution $(x,v)$ of the ODE system exists by the Picard-Lindel\"off Theorem. Since $\dot x = v$ is $C^1$-smooth, we see that the solution $x$ of the heavy ball ODE is $C^2$-smooth.

{\bf Step 1: $x_t$ remains in $\mathcal U_\alpha$.}
Note that
\[ 
\frac{d}{dt}\left(f(x_t) + \frac12\|\dot x_t\|^2\right) = \langle \nabla f(x_t), \dot x_t\rangle + \langle \dot x_t, \ddot x_t\rangle = \langle \ddot x + \nabla f(x),\dot x\rangle = -2\sqrt\mu\,\|\dot x_t\|^2 \leq 0,
\]
so
\[
f(x_t) \leq f(x_t) + \frac12\|\dot x_t\|^2 \leq f(x_0) + \frac12\|\dot x_0\|^2 = f(x_0) < \alpha
\]
for all $t \geq 0$ since $\dot x_0 = 0$, which implies $x_t \in \mathcal U_\alpha$ for all $t \geq 0$.

{\bf Step 2: Bounding $f(x_t)$.}
Let $z_t:= \pi(x_t)$ denote the closest point projection of $x_t$ onto $\mathcal{M}$ and by $\dot z_t$ its derivative. Consider the Lyapunov function
\begin{align*}
\L(t) &= f(x_t) - f(\pi(x_t)) + \frac{1}{2}\|\dot{x}_t + \sqrt{\mu}(x_t - \pi(x_t))\|^2
\end{align*}
We will show that $\L'(t) \leq -\sqrt{\mu}\L(t)$ under some neighborhood assumption on $x_t$.
Using the heavy ball dynamics and properties of the projection, we can bound $\L'(t)$ as:
\begin{align*}
\L'(t) &= \langle \nabla f(x_t), \dot{x}_t \rangle + \big\langle \dot{x}_t + \sqrt{\mu}(x_t - \pi(x_t)), \ddot{x}_t + \sqrt{\mu}(\dot{x}_t - \dot z_t)\big\rangle \\
    &= \big\langle\dot x_t,\: \nabla f(x_t) + \ddot x_t + \sqrt\mu \,\dot x_t - \sqrt\mu\, \dot z_t\big\rangle + \sqrt{\mu}\,\big\langle x_t - \pi(x_t),\: \ddot x_t + \sqrt\mu \,\dot x_t - \sqrt\mu\dot z_t\big\rangle\\
    &= \big\langle \dot x_t, -\sqrt\mu\,\dot x_t - \sqrt{\mu}\dot z_t,\big\rangle + \sqrt\mu \,\big\langle x_t-z_t,\:   -\sqrt\mu\, \dot x_t - \nabla f(x_t) - \sqrt{\mu}\,\dot z_t \big\rangle\\
    &\leq - \sqrt\mu \|\dot x_t\|^2- \mu\langle x_t-z_t, \dot x_t\rangle - \sqrt\mu\,\langle \nabla f(x_t), \:x_t-z_t\rangle
\end{align*}
where we used the heavy ball dynamics $\ddot{x}_t = -2\sqrt{\mu}\dot{x}_t - \nabla f(x_t)$ and the geometric properties of the closest point projection:
\begin{enumerate}
    \item $-\langle \dot x_t, \,\dot z_t\rangle \leq 0$ by Lemma \ref{lemma auxiliary parallel movement} and
    \item $\langle x_t-z_t, \,\dot z_t\rangle =0$ since $x_t-z_t$ meets $\M$ orthogonally at $z_t$ and $\dot z_t$ is tangent to $\M$ at $z_t$.
\end{enumerate}
Next, using the $\mu$-strong aiming condition, we have
\begin{equation*}
\langle \nabla f(x_t), x_t - \pi(x_t) \rangle \geq f(x_t) - f(\pi(x_t)) + \frac{\mu}{2}\|x_t - \pi(x_t)\|^2.
\end{equation*}
Substituting this into the bound on $\L'(t)$ and simplifying gives:
\begin{align*}
\L'(t) &\leq -\sqrt\mu \,\|\dot x_t\|^2 - \mu\,\langle x_t - \pi(x_t), \dot x_t\rangle - \sqrt{\mu}\left(f(x_t) - f(\pi(x_t)) + \frac{\mu}{2}\|x_t - \pi(x_t)\|^2\right)\\
    &= - \sqrt\mu \left(f(x_t) - f(\pi(x_t)) + \frac12 \left\|\dot x_t + \sqrt\mu\big(x_t-\pi(x_t)\big)\right\|^2 + \frac12\,\|\dot x_t\|^2\right)\\
    &\leq - \sqrt \mu\,\L(t).
\end{align*}

Now, we can bound the initial value of the Lyapunov function $\L(0)$ as follows:
\begin{align*}
\L(0) &= f(x_0) -  \inf_{z\in\R^d} f(z) + \frac{1}{2}\|\dot{x}_0 + \sqrt{\mu}(x_0 - \pi(x_0))\|^2 \\
&= f(x_0) -  \inf_{z\in\R^d} f(z) + \frac{\mu}{2}\text{dist}(x_0, \mathcal{M})^2
\end{align*}
since $\dot{x}_0 = 0$ and $\dist(x, \mathcal M) = \|x-\pi(x)\|$.
We deduce that
\[
\frac{d}{dt} e^{\sqrt\mu t}\L(t) = \big(\sqrt\mu \L(t) + \L'(t)\big) e^{\sqrt\mu t} \leq 0
\]
so
\begin{align*}
f(x_t) - \inf f &\leq \L(t) \leq e^{-\sqrt{\mu}t}\L(0) 
= e^{-\sqrt{\mu}t}\left(f(x_0) - \inf f + \frac{\mu}{2}\text{dist}(x_0, \mathcal{M})^2\right) \qedhere
\end{align*}
\end{proof}

Next, we prove the `global convergence' statement.
\globalconvergence*
\begin{proof}%
    The idea of the proof is to show that at some large time $T$, the trajectory of $x$ enters the set $\mathcal U_\alpha$ with sufficiently low velocity that it gets trapped in $\mathcal U_\alpha$. From that point onwards, the proof of Theorem \ref{theorem continuous strongly convex} applies with minor modifications.

    {\bf Step 1.} Denote by $E(t) = f(x_t) + \frac12\,\|\dot x_t\|^2$ the `total energy' of the curve $x$ at time $t$. As in the proof of Theorem \ref{theorem continuous strongly convex}, we find that $E'(t) = - 2\sqrt\mu\,\|\dot x_t\|^2$, so in particular $\|\dot x\|^2$ is square integrable in time:
    \[
    \int_0^\infty \|\dot x\|^2 \dt = \frac{E(0) - \lim_{t\to\infty} E(t)}{2\sqrt\mu} \leq \frac{f(x_0) - \inf f}{2\sqrt\mu}.
    \]
    Recall for future use that $f(x_t) \leq E(t) \leq E(0) = f(x_0)$ for $t\geq0$.

    {\bf Step 2.}
    In this step, we show that also $\nabla f(x_t)$ is square integrable in time. To do this, we first observe that 
    \begin{align*}
    \int_0^T \|\nabla f(x) + 2\sqrt\mu\,\dot x\|^2 \dt 
        &= \int_0^T \|\nabla f(x)\|^2+4\sqrt\mu\,\langle \nabla f(x), \dot x\rangle + \|\dot x\|^2\dt\\
        &= \int_0^T \|\nabla f(x)\|^2\dt +4\sqrt\mu\,\int_0^T \frac{d}{dt} f(x_t)\dt +4\mu\int_0^T \|\dot x\|^2\dt.
    \end{align*}
    On the other hand, we can write
    \begin{align*}
    \int_0^T \|\nabla f(x) + 2\sqrt\mu\,\dot x\|^2 \dt
        &= -\int_0^T \langle \nabla f(x) + 2\sqrt\mu\,\dot x, \:\ddot x\rangle \dt\\
        &= -\int_0^T \langle \nabla f(x), \ddot x\rangle \dt - \sqrt\mu \int_0^T \frac{d}{dt}\|\dot x\|^2\dt\\
        &= -\langle \nabla f(x_T), \dot x_T\rangle +\int_0^T \langle D^2f(x)\dot x, \dot x\rangle \dt - \sqrt \mu \,\|\dot x_T\|^2
    \end{align*}
    since $\dot x_0 = 0$. Overall, we find that
    \begin{align*}
    \int_0^T \|\nabla f(x)\|^2\dt
        &= 4\sqrt\mu\big(f(x_0) - f(x_T)\big) + \langle \nabla f(x_T), \dot x_T\rangle - \sqrt\mu\,\|\dot x_T\|^2 \\
        &\quad + \int_0^T \langle D^2f(x)\dot x, \dot x\rangle - 4\mu\,\|\dot x\|^2\dt\\
        &\leq 4\sqrt\mu\big(f(x_0)-\inf f\big) + L\int_0^T\|\dot x\|^2 \dt + \langle \nabla f(x), \dot x\rangle(T)
    \end{align*}
    where $L= L_{f(x_0)}$ is the Lipschitz-constant of $\nabla f$ on the set $\{x : f(x) <f(x_0)\}$. Note that
    \[
    t\mapsto \langle \nabla f(x), \dot x\rangle = \frac{d}{dt} f(x_t)
    \]
    is the derivative of the bounded function $f(x_t) \in [\inf f, f(x_0)]$ and continuous, so there exists a sequence of times $T_n\to\infty$ such that $\langle f(x), \dot x\rangle(T_n)\to0$. Since $\|\nabla f\|^2$ is a non-negative integrand, we can bound
    \begin{align*}
    \int_0^\infty \|\nabla f(x)\|^2\dt
        &\leq \lim_{n\to\infty} \left(4\sqrt\mu\big(f(x_0)-\inf f\big) + L\int_0^{T_n}\|\dot x\|^2 \dt + \langle \nabla f(x), \dot x\rangle(T_n)\right)\\
        &= 4\sqrt\mu\big(f(x_0)-\inf f\big) + L\int_0^{\infty}\|\dot x\|^2 \dt < +\infty.
    \end{align*}

    {\bf Step 3.} Using Steps 1 and 2, we find that 
    \[
    \int_0^T \|\nabla f(x_t)\|^2 + \|\dot x_t\|^2 \dt < +\infty.
    \]
    In particular, there exists a sequence of times $t_n\to \infty$ such that 
    \[
    \|\nabla f(x(t_n))\|^2 + \|\dot x(t_n)\|^2 \to 0
    \]
    as $n\to\infty$. We can therefore choose $T>0$ such that
    \[
    \|\nabla f(x_T)\|^2 + \|\dot x_T\|^2 < \min\big\{\delta, \alpha\}.
    \]
    Then we find that
    \[
    \|\nabla f(x_T)\|^2 < \delta \quad\Ra\quad f(x_T) < \alpha, \qquad \|\dot x_T\|^2 < \alpha \quad \Ra\quad  E(T) = f(x_T) + \frac12\,\|\dot x_T\|^2 < \alpha.
    \]
    In particular, we conclude that $f(x_t) \leq E(t) < \alpha$ for all $t> T$, i.e.\ $x_t \in \mathcal U_\alpha$ for all $t>T$. Thus, by the same argument as Theorem \ref{theorem continuous strongly convex}, we find that
    \[
    \mathcal L(t) := f(x_t) -\inf f + \frac12\big\|\dot x + \sqrt\mu\,\big(x_t - \pi(x_t)\big)\big\|^2
    \]
    satisfies $\mathcal L(t) \leq e^{-\sqrt\mu(t-T)} L(T)$ for $t>T$, so
    \begin{align*}
    f(x_t) - \inf f &\leq \mathcal L(t) \leq e^{\sqrt\mu(T-t)}\left(f(x_T) -\inf f + \frac12 \big\|\dot x_T + \sqrt\mu(x_T-\pi(x_T))\big\|^2\right)\\
    &\leq e^{\sqrt\mu(T-t)}\left(f(x_t) -\inf f + \frac22\big\{\|\dot x_T\|^2 + \|x_T-\pi(x_T)\|^2\big\}\right)\\
    &\leq e^{\sqrt\mu(T-t)}\left(\frac{3\alpha}2 + \dist(x_T, \mathcal M)^2\right).\qedhere
    \end{align*}
    \end{proof}

    \begin{remark}
        Obviously, the condition that $\nabla f$ is Lipschitz-continuous on all sublevel sets could easily be relaxed to requiring that the initialization $x_0$ is such that $\nabla f$ is merely Lipschitz-continuous on the set $\{x : f(x) < f(x_0)\}$, or even on every connected component of the set.
    \end{remark}

\subsection{Proof of Theorem \ref{theorem discrete}: Acceleration in Discrete Time (Deterministic setting)}\label{appendix acceleration discrete}

\smallnegeigen*
\begin{proof}%
The function $f(x) + \frac\eps2 \|x\|^2$ is convex since all eigenvalues of its Hessian $D^2f + \eps I$ are non-negative, so
\[
f(z) + \frac\eps2\,\|z\|^2 \geq f(x) + \frac\eps2 \,\|x\|^2 + \langle \nabla f(x) + \eps x, z-x\rangle 
\]
which is equivalent to
\begin{align*}
\langle \nabla f(x), x-z\rangle &\geq f(x) - f(z) + \frac\eps2 \,\|x\|^2+ \eps \langle x, z-x\rangle -\frac\eps2\, \|z\|^2\\
	&= f(x) - f(z) + \frac\eps2 \,\|x\|^2+ \eps\langle x, z\rangle - \eps \|x\|^2 -\frac\eps2\, \|z\|^2\\
	&=  f(x) - f(z) - \frac\eps2 \|z-x\|^2.\qedhere
\end{align*}
\end{proof}

Before proving Theorem \ref{theorem discrete}, we recall a well-known auxiliary result.

\begin{lemma}\citep[Lemma 13]{gupta2023achieving}\label{lemma gd}
    Assume that $\nabla f$ is Lipschitz-continuous with Lipschitz-constant $L$. Then
    \[
    f(x-\eta g) \leq f(x) - \eta\langle \nabla f(x), g\rangle + \frac{L\eta^2}2\,\|g\|^2.
    \]
\end{lemma}

\discrete*
\begin{proof}%
{\bf Setup.} Denote by $ P^\bot = I -  P$ the orthogonal projection onto the orthogonal complement of the space which $ P$ projects onto. Consider the Lyapunov sequence defined by
\begin{align*}
    \L_n = f(x_n) - \inf f + \frac12\|\  P^\bot v_n + \sqrt\mu (x_n' -\pi(x_n')) \|^2 + \frac{(1+\sqrt{\mu\eta})^2}{2(1-\sqrt{\mu\eta})}\| P v_n\|^2.
\end{align*}
This is a variation of the usual Lyapunov sequence in which we separately analyze the tangential and normal velocities. Note however that 
\[
\frac{(1+\sqrt{\mu\eta})^2}{(1-\sqrt{\mu\eta})} = 1 + O(\sqrt\eta),
\]
i.e.\ if $\eta\to0$, we recover the Lyapunov function in the continuous time setting where tangential and normal velocity are not separated:
\[
 \|\  P^\bot v + \sqrt\mu(x -\pi(x)) \|^2 + \frac{(1+\sqrt{\mu\eta})^2}{2(1-\sqrt{\mu\eta})}\| P v\|^2 \to \|\  P^\bot v + \sqrt\mu(x -\pi(x)) \|^2 + \| P v\|^2 = \|v + \sqrt\mu(x-\pi(x))\|^2
\]
as $\eta\to 0$ since the vectors $ P v$ and $ P^\bot v + x -\pi(x)$ are orthogonal and $v=  P v +  P^\bot v$.
We want to show that $\L_{n+1} \leq (1-\sqrt{\mu\eta})\L_n$. Note that $1-\sqrt{\mu\eta}\geq 1-\sqrt{\frac\mu L} \geq 0$ since $\mu\leq L$ by Lemma \ref{lemma geometric implications}.

For simplicity, we assume that $x^*=0$, i.e.\ $x_n' - \pi(x_n') = x_n' -  P x_n' =  P^\bot x_n'$.

{\bf Step 1.} Since $f$ is $L$-smooth and $\eta\leq 1/L$, we have
\[
f(x_{n+1}) \leq f(x_n') - \left(1-\frac{L\eta}2\right)\eta\,\|\nabla f(x_n')\|^2 = f(x'_n) - \frac {\eta}2\,\|\nabla f(x'_n)\|^2
\]
by Lemma \ref{lemma gd} with $x=x_n'$ and $g=\nabla f(x_n')$.

{\bf Step 2.} 
We compute
\begin{align*}
 P^\bot v_{n+1} &+ \sqrt\mu  P^\bot x'_{n+1}
	=  P^\bot v_{n+1} +  \sqrt\mu  P^\bot(x'_n + \sqrt\eta v_{n+1}- \eta\df)\\
	&= ( P^\bot+ \sqrt{\mu\eta}  P^\bot) v_{n+1} + \sqrt\mu P^\bot x'_n - \eta\sqrt\mu\, P^\bot \df\\
	&= \rho(1+\sqrt{\mu\eta}) P^\bot v_n + \sqrt\mu  P^\bot x'_n - \sqrt\eta\big(\sqrt{\mu\eta}+\rho(1+\sqrt\mu \sqrt\eta) \big)\, P^\bot \df \\
 &= (1-\sqrt{\mu\eta}) P^\bot v_n + \sqrt\mu  P^\bot x'_n - \sqrt\eta\, P^\bot \df,
\end{align*}
where we simplify the coefficients in the last step by substituting $\rho = \frac{1-\sqrt{\mu\eta}}{1+\sqrt{\mu\eta}}.$ So,
\begin{align*}
\frac12 \big\| P^\bot v_{n+1} + \sqrt\mu  P^\bot x'_{n+1}\big\|^2
	&= \frac{(1 - \sqrt{\mu\eta})^2}2\| P^\bot v_n\|^2 +\frac\mu2 \| P^\bot x'_n\|^2 + \frac{\eta}2 \| P^\bot \df\|^2 \\
 &\quad + \sqrt \mu(1 - \sqrt{\mu\eta}) \langle  P^\bot v_n, \, P^\bot x'_n\rangle - \sqrt{\mu\eta}\,\langle  \df,  P^\bot x'_n\rangle\\
	&\quad - \sqrt\eta(1-\sqrt{\mu\eta})\langle \df,  P^\bot v_n\rangle.
\end{align*}
Recall that since both $ P$ and $ P^\bot$ are orthogonal projections, for any $x,y \in \R^d$, $\langle  P x,  P y \rangle = \langle  P x, y \rangle = \langle x,   P y \rangle$, and the analogous result holds for $ P^\bot$ as well.

{\bf Step 3.} Now we expand the last term in the Lyapunov sequence,
\begin{align*}
    \frac{(1+\sqrt{\mu\eta})^2}{2(1-\sqrt{\mu\eta})}\|  P v_{n+1} \|^2 &= \frac{(1+\sqrt{\mu\eta})^2}{2(1-\sqrt{\mu\eta})} \rho^2\|  P(v_n - \sqrt\eta \df )  \|^2\\
    &= \frac{1-\sqrt{\mu\eta}}{2}\left( \| P v_n\|^2 + \eta\| P \df\|^2 - 2 \sqrt\eta \langle \df,  P v_n\rangle \right) 
    .%
\end{align*}

{\bf Step 4.} We add the expressions from the previous steps and use the fact that $ P v_n +  P^\bot v_n = v_n$ to get
\begin{align*}
\L_{n+1}
	&= \frac{(1 - \sqrt{\mu\eta})^2}2\| P^\bot v_n\|^2 +\frac\mu2 \| P^\bot x'_n\|^2 + \frac{\eta}2 \| P^\bot \df\|^2 \\
 &\quad + \sqrt \mu(1 - \sqrt{\mu\eta}) \langle  P^\bot v_n, \, P^\bot x'_n\rangle - \sqrt{\mu\eta}\,\langle  \df,  P^\bot x'_n\rangle\\
	&\quad - \sqrt\eta(1-\sqrt{\mu\eta})\langle \df, v_n\rangle + \frac{1-\sqrt{\mu\eta}}{2} \| P v_n\|^2\\
 &\quad + \frac{\eta(1-\sqrt{\mu\eta})}{2}\| P \df\|^2 + f(x_{n+1}) - \inf f. 
\end{align*}

Using \eqref{eq assumption convexity wrt minimizer} and \eqref{eq assumption mild non-convexity} to bound the inner products $\langle \df, v_n\rangle$ and $\langle  \df,  P^\bot x'_n\rangle$,
\begin{align*}
\L_{n+1}
	&\leq \frac{(1 - \sqrt{\mu\eta})^2}2\| P^\bot v_n\|^2 +\frac\mu2 \| P^\bot x'_n\|^2 + \frac{\eta}2 \| P^\bot \df\|^2 \\
 &\quad + \sqrt \mu(1 - \sqrt{\mu\eta}) \langle  P^\bot v_n, \, P^\bot x'_n\rangle 
 - \sqrt{\mu\eta}\, \left( f(x'_n) - \inf f + \frac \mu 2 \| P^\bot x'_n \|^2 \right)  \\
	&\quad - (1-\sqrt{\mu\eta}) \left( f(x'_n) - f(x_n) - \frac{\eps}2 \| \sqrt\eta\,v_n \|^2 \right)
 + \frac{1-\sqrt{\mu\eta}}{2} \| P v_n\|^2\\
 &\quad + \frac{\eta(1-\sqrt{\mu\eta})}{2}\| P \df\|^2 + f(x'_n) - \frac{\eta}{2}\|\df\|^2 - \inf f. 
\end{align*}
Now, we use Pythagoras theorem, i.e.\ for all $w\in \R^d$, $\| P w\|^2 + \|  P^\bot w \|^2 = \| P w +  P^\bot w \|^2 = \|w\|^2$, and rearrange some of the terms,
\begin{align*}
    \L_{n+1}
	&\leq \frac{(1 - \sqrt{\mu\eta})^2 + (1 - \sqrt{\mu\eta})\eta\eps}2\| P^\bot v_n\|^2 + \frac{(1-\sqrt{\mu\eta})(1+\eta\eps)}{2} \| P v_n\|^2\\
 &\quad + \frac{\mu(1-\sqrt{\mu\eta})}2 \| P^\bot x'_n\|^2 + \frac{\eta-\eta}2 \| \df\|^2 - \frac{\sqrt\eta^3\sqrt\mu}{2}\| P \df\|^2 \\
 &\quad + \sqrt \mu(1 - \sqrt{\mu\eta}) \langle  P^\bot v_n, \, P^\bot x'_n\rangle \\
&\quad + (1-\sqrt{\mu\eta} - 1 + \sqrt{\mu\eta}) f(x'_n) + (1-\sqrt{\mu\eta})(f(x_n) - \inf f).
\end{align*}
The coefficients of $f(x'_n)$ and $\|\df\|^2$ are zero and the coefficient of $\| P \df \|^2$ is negative, so we can disregard those terms. If $\eps \leq \sqrt{\mu/\eta}$, the coefficient $\| P^\bot v_n\|^2$ can be bounded as
\begin{align*}
     \frac{(1 - \sqrt{\mu\eta})(1 - \sqrt{\mu\eta} + \eta\eps)}2 \leq \frac{(1 - \sqrt{\mu\eta})(1 - \sqrt{\mu\eta} + \sqrt{\mu\eta})}2
     \leq \frac{1 - \sqrt{\mu\eta}}2,
\end{align*} and the coefficient of $\| P^\bot v_n\|^2$ can be bounded as
\begin{align*}
    \frac{(1-\sqrt{\mu\eta})(1+\eta\eps)}{2} \leq \frac{(1-\sqrt{\mu\eta})(1+\sqrt{\mu\eta})}{2} \leq \frac{(1 +\sqrt{\mu\eta})^2}{2}.
\end{align*}
Thus, we conclude that
\begin{align*}
    \L_{n+1}
	&\leq (1-\sqrt{\mu\eta})(f(x_n) - \inf f) + \frac{1 - \sqrt{\mu\eta}}2\| P^\bot v_n\|^2 \\
 &\quad + \frac{\mu(1-\sqrt{\mu\eta})}2 \| P^\bot x'_n\|^2 + \sqrt\mu(1 - \sqrt{\mu\eta}) \langle  P^\bot v_n, \, P^\bot x'_n\rangle \\
&\quad + \frac{(1 +\sqrt{\mu\eta})^2}{2} \| P v_n\|^2 \\
& = (1-\sqrt{\mu\eta})\L_n.
\end{align*}
Consequently,
\[
 f(x_n) - \inf f \leq \L_n \leq (1-\sqrt{\mu\eta})^n L_0 = (1-\sqrt{\mu\eta})^n\left[ f(x_0) - \inf f + \frac\mu2 \| P^\bot x_0\|^2 \right].\qedhere
\]
\end{proof}

\subsection{Proof of Theorems \ref{theorem stochastic additive} and \ref{theorem decreasing learning rate}: Acceleration with additive noise}

The proof of Theorem \ref{theorem stochastic additive} mimics that of Theorem \ref{theorem discrete} with minor modifications to account for stochastic gradient estimates. Since $\omega_n$ is stochastically independent of anything which has happened in the algorithm before, in particular of $x_n', v_n$, we have the following.

\begin{lemma}\label{lemma stochastic identities}
For all $n\in\mathbb N$, we have
    \begin{align*}
        \E\big[\langle g(x_n', \omega_n), \nabla f(x_n')\rangle\big] &= \E\big[\|\nabla f(x_n')\|^2]\\
        \E\big[\langle g(x_n', \omega_n), v_n\rangle\big] &= \E\big[\langle \nabla f(x_n'), v_n\rangle]\\
        \E\big[\langle g(x_n', \omega_n),  P^\bot x_n'\rangle\big] &= \E\big[\langle \nabla f(x_n'),  P^\bot x_n'\rangle]\\
        \E\big[\|g_n\|^2\big] &= \E\big[\|\nabla f(x_n')\|^2\big] + \E\big[\|g_n - \nabla f(x_n')\|^2\big]
    \end{align*}
    where the expectations are taken over the (potentially random) initial condition $x_0$ as well as the random coefficients $\omega_0, \dots, \omega_n$ which govern the gradient estimates.
\end{lemma}

A proof can be found in \citep[Lemma 15]{gupta2023achieving}. The second identity, which is not included therein, can be proved analogously.
As an application, we prove a stochastic analogue of Lemma \ref{lemma gd}.

\begin{lemma}\label{lemma sgd}
    Assume that $f$ is $L$-smooth and $\E\big[\|g(x,\omega)-\nabla f(x)\|^2\big] \leq \sigma_a^2 + \sigma_m^2\|\nabla f(x)\|^2$ for all $x\in\R^d$. Then, if $x, \omega$ are independent random variables, we have
    \[
    \E_{(x,\omega)}\big[f(x-\eta g(x,\omega))\big] \leq \E\big[f(x)\big] - \left(1- \frac{L(1+\sigma^2)\eta}2\right)\eta \,\E\big[\|\nabla f(x)\|^2\big] + \frac{L\eta^2}2\,\sigma_a^2.
    \]
\end{lemma}

\begin{proof}
Recall that for any $x, \eta, g$, the following holds
\begin{align*}
    f(x- \eta g) &\leq f(x) - \eta\langle \nabla f(x), g\rangle + \frac{L\eta^2}2\,\|g\|^2
\end{align*}
by Lemma \ref{lemma gd}.
We now assume that $g= g(x,\omega)$ is a random estimator for $\nabla f(x)$, where $x$ may be random, but $\omega$ is independent of $x$. Then, by Lemma \ref{lemma stochastic identities}, we have
\begin{align*}
    \E\big[f(x-\eta g)\big]&\leq \E\big[f(x)\big] - \eta \,\E\big[\|\nabla f(x)\|^2\big] + \frac{L\eta^2}2\,\E\big[\|g\|^2\big]\\
    &=\E\big[f(x)\big] - \eta \,\E\big[\|\nabla f(x)\|^2\big] + \frac{L\eta^2}2\,\big\{\E\big[\|\nabla f(x)\|^2 + \E\big[\|g-\nabla f(x)\|^2\big]\big\}\\
    &\leq \E\big[f(x)\big] - \left(1-\frac{L\eta}2\right)\eta \,\E\big[\|\nabla f(x)\|^2\big] + \frac{L\eta^2}2\big\{\sigma_a^2 + \sigma_m^2\,\E\big[\|\nabla f(x)\|^2\big]\\
    &= \E\big[f(x)\big] - \left(1-\frac{L(1+\sigma_m^2)\eta}2\right)\eta \,\E\big[\|\nabla f(x)\|^2\big] + \frac{L\eta^2}2\,\sigma_a^2.\qedhere
\end{align*}
\end{proof}

Finally, we provide an auxiliary result to resolve a recursion.

\begin{lemma}\label{lemma recursion}
    Assume a sequence $x_n$ satisfies the recursive estimate $x_{n+1} \leq a x_n + b$ for $a\in (0,1)$ and $b\geq 0$. Then
    \[
    x_n \leq a^n x_0 + \frac{b}{1-a}.
    \]
\end{lemma}

\begin{proof}
    Consider the sequence $y_n:= x_n+ \frac{b}{a-1}$. Then
    \begin{align*}
    y_{n+1} &= x_{n+1} + \frac{b}{a-1} \leq ax_n + b+ \frac{b}{a-1} = ax_n + \frac{a-1}{a-1}b + \frac{b}{a-1} = ax_n + \frac{a b}{a-1}\\
    &= a\left(x_n+\frac{b}{a-1}\right)= ay_n.
    \end{align*}
    In particular, $y_{n} \leq a^n y_0$, so
    \[
    x_n = y_n + \frac b{1-a} \leq a^n\left(x_0 + \frac{b}{a-1}\right) + \frac{b}{1-a}.
    \]
    If $b>0$, the estimate follows since $b/(a-1)<0$.
\end{proof}

We now prove the first main result of this section.

\additive*
\begin{proof}
The proof is mostly identical to that of Theorem \ref{theorem discrete} with minor modifications: In Step 1, we obtain the estimate
\[
\E[f(x_{n+1}]\leq \E[f(x_n')] - \frac{\eta}2 \,\E\big[\|\nabla f(x_n')\|^2\big] + \frac{L\sigma_a^2\eta^2}2
\]
by Lemma \ref{lemma sgd}. In the quadratic terms in steps 2 and 3, we need to take the expectation of $g_n$ rather than $\nabla f(x_n')$. By construction, the terms involving $\E[\|\nabla f(x_n')\|^2]$ still balance with the same parameters, leading to an additional contribution of $\eta\,\sigma_a^2$ due to the stochastic gradient estimates. Since $\eta\leq 1/L$, we can bound $L\eta^2\leq \eta$ so overall we obtain the estimate
\[
\L_{n+1} \leq (1-\sqrt{\mu\eta})\L_n + \sigma_a^2\eta.
\]
in place of $\L_{n+1} \leq (1-\sqrt{\mu\eta})\L_n$ since all other terms are identical under the expectation using Lemma \ref{lemma stochastic identities}. The claim now follows from Lemma \ref{lemma recursion}.
\end{proof}

For readers who are looking for a more detailed proof, we note that Theorem \ref{theorem stochastic additive} is a special case of Theorem \ref{theorem agnes} with $\sigma_m=0$, where we provide a full proof. 

In the same spirit, we sketch the proof of Theorem \ref{theorem decreasing learning rate}. Let us recall the statement.

\decreasingstep*

\begin{proof}
Note that in the proof of Theorem \ref{theorem agnes}, we build on the relationships
\begin{align*}
f(x_{n+1}) &\leq f(x_n') - \frac{\eta_n}2 \,\|\nabla f(x_n')\|^2\\
x_{n+1}' &= x_n' - \eta_n g_n + \sqrt{\alpha_n}v_n\\
v_{n+1} &= \rho_n(v_n - \sqrt{\alpha_n}g_n)
\end{align*}
and do not enter further into the recursion. We additionally note that if $\eta_n$ is a {\em monotone decreasing} sequence, then the sequence $\lambda_{n+1} := \frac{(1+ \sqrt{\mu\eta_n})^2}{1-\sqrt\mu \eta_n}$ is also monotone decreasing. 
Hence, if $\eta_n\leq \frac1{L(1+\sigma_m^2)}$ for all $n\in\N$, then by the same proof as Theorem \ref{theorem stochastic additive}, the sequence
\[
\L_{n+1} := f(x_n) -\inf f + \frac12\big\| P^\bot v_n + \sqrt\mu\, P^\bot x_n'\big\|^2 + \frac{\lambda_n}2 \,\| P v_n\|^2
\]
satisfies
\begin{align*}
\L_{n} &\leq \big(1-\sqrt{\mu\eta_n})\big)\left\{ f(x_n) -\inf f + \frac12\big\| P^\bot v_n + \sqrt\mu\, P^\bot x_n'\big\|^2 + \frac{\lambda_{n+1}}2 \,\| P v_n\|^2\right\} + \sigma_a^2\eta_n\\
    &\leq \big(1-\sqrt{\mu\eta_n})\big)\L_n + \sigma_a^2\eta_n
\end{align*}
if the parameters are chosen as in the theorem statement. If specifically $\eta_n = \frac{1}{\mu\,(n+n_0+1)^2}$ for $n_0 = \sqrt{L/\mu}$, then
\[
\L_{n+1} \leq \left(1- \frac1{n+n_0+1}\right) \L_n + \frac{\sigma_a^2}{\mu(n+n_0+1)^2} = \frac{n+n_0}{n+n_0+1}\L_n  + \frac{\sigma_a^2}{\mu(n+n_0+1)^2}.
\]
Thus the sequence $z_n:= (n+n_0)\L_n$ satisfies the relation
\[
z_{n+1} = (n+n_0+1)\L_{n+1} \leq (n+n_0)\L_n + \frac{\sigma_a^2}{\mu(n+n_0+1)} = z_n + \frac{\sigma_a^2}{\mu(n+n_0+1)}
\]
so
\begin{align*}
z_n &= z_0 + \sum_{k=1}^n (z_k-z_{k-1}) \leq z_0 + \sum_{k=0}^{n-1}\frac{\sigma_a^2}{\mu(k+n_0+1)} \leq n_0\L_0 + \frac{\sigma_a^2}{\mu}\int_0^{n-1} \frac1{n_0+t}\dt\\
    &\leq n_0\L_0 + \frac{\sigma_a^2}{\mu}\log\left( 1+ \frac{n-1}{n_0}\right).
\end{align*}
Overall, we find that
\begin{align*}
\E\big[f(x_n) - \inf f\big] &\leq \L_n = \frac{z_n}{n+n_0}\\
    &\leq \frac{\sqrt{\frac L\mu}\,\E\left[f(x_0) - \inf f + \frac12\|x_0-\pi(x_0)\|^2\right] + \frac{\sigma_a^2}{\mu}\,\log\big(1+\frac{n-1}{n_0}\big)}{n+n_0}.
\end{align*}
\end{proof}

\subsection{Proof of Theorem \ref{theorem agnes}: Acceleration with both additive and multiplicative noise}\label{appendix agnes}

In the strongly convex setting, \cite{gupta2023achieving} state a version of Theorem \ref{theorem agnes} in slightly greater generality in terms of choosing variables. The same more general proof goes through also here after we account for the tangential and normal components of the velocity as in the proof of Theorem \ref{theorem discrete}.

\agnes*
\begin{proof}
	{\bf Setup.} Mimicking the proof of Theorem \ref{theorem discrete}, consider the sequence
	\[
	\L_n = \E\big[ f(x_n) - f(x^*)\big] + \frac12\,\E\left[\big\| b\, P^\bot v_n + a\big(x_n' - \pi(x_n')\big)\|^2\right] + \frac\lambda2\,\E\left[ \|  P v_n \|^2 \right]
	\]
	for constants
	\[
	b = \sqrt{\frac{(1+\sigma_m^2)\alpha}\eta}, \qquad \lambda = \frac{(b+\sqrt{\mu}\gamma)^2}{b-\sqrt\mu\,\gamma}\,\frac{\gamma}{\sqrt\alpha}\qquad\text{where} \quad\gamma = \sqrt\mu(\eta-\alpha)+b\sqrt\alpha.
	\]
	The constants will be motivated below where they are introduced. Note that we have $\eta =\alpha$ if $\sigma_m=0$ and thus $b=1$ and $\gamma = \sqrt\alpha$, recovering the situation of Theorem \ref{theorem stochastic additive}.
	We want to show that $\L_{n+1} \leq (1-\sqrt\mu\,\sqrt\alpha/b) \L_n$. For simplicity, we again assume without loss of generality that $\pi(x) =  P x$, i.e.\ $x^*=0$ throughout the proof.
	
	{\bf Step 1.} Consider the first term first. Note that
	\begin{align*}
		\E\big[ f(x_{n+1})\big] &= \E\big[ f(x'_n - \eta g_n)\big]
		\leq \E\big[ f(x'_n)\big] -  \left(1-\frac {L(1+\sigma^2)}2\eta\right)\eta\,\E\big[\|\nabla f(x'_n)\|^2\big] + \frac{L\eta^2}2\,\sigma_a^2\\
		&\leq \E\big[ f(x'_n)\big] -  \frac\eta2\,\E\big[\|\nabla f(x'_n)\|^2\big] + \frac{L\eta^2}2\,\sigma_a^2
	\end{align*}
	if $\eta\leq \frac1{L(1+\sigma^2)}$ by Lemma \ref{lemma sgd}.
	
	{\bf Step 2.} We now turn to the second term and use the definition of $x'_{n+1}$ from (\ref{eq agnes}),
	\begin{align*}
		b P^\bot v_{n+1} + \sqrt\mu P^\bot x'_{n+1} &= b P^\bot v_{n+1} + \sqrt\mu P^\bot( x_{n}' + \sqrt\alpha\,v_{n+1} - \eta g_n )\\
		&= (b+\sqrt{\mu\alpha})\rho  P^\bot\big(v_n - \sqrt\alpha g_n) + \sqrt\mu P^\bot x_n' - \sqrt\mu\eta  P^\bot g_n\\
		&= (b+\sqrt{\mu\alpha})\rho P^\bot v_n + \sqrt\mu P^\bot x_n' - \big(\sqrt\mu\eta+ \rho(b+\sqrt{\mu\alpha})\sqrt\alpha\big) P^\bot g_n.
	\end{align*}
	In analogy to the proof of Theorem \ref{theorem discrete}, we have $\rho = \frac{b-\sqrt{\mu\alpha}}{b+\sqrt{\mu\alpha}}$, so
	\[
	b P^\bot v_{n+1} + \sqrt\mu P^\bot x'_{n+1} = (b-\sqrt{\mu\alpha}) P^\bot v_n + \sqrt\mu P^\bot x_n' - \big(\sqrt\mu\eta + (b-\sqrt{\mu\alpha})\sqrt\alpha\big)g_n.
	\]
	In the deterministic case where $\eta=\alpha$, the coefficient $\sqrt\mu(\eta-\alpha)+b\sqrt\alpha$ of $g_n$ simplified to $b\sqrt\alpha$. It does not in this more general setting anymore, so we introduce a new notation: $\gamma = \sqrt\mu(\eta-\alpha) + b\sqrt\alpha$.
	
	Taking expectation of the square, we find that
	\begin{align*}
		\E&\left[\big\|b P^\bot v_{n+1} + \sqrt\mu\, P^\bot x'_{n+1}\big\|^2\right] \\
		&= (b-\sqrt{\mu\alpha})^2\, \E\big[\| P^\bot v_n\|^2\big] + 2\sqrt\mu(b-\sqrt{\mu\alpha}) \,\E\big[ \langle  P^\bot v_n,  P^\bot x'_n\rangle\big] + a^2\E\big[\| P^\bot x'_n\|^2\big]\\
		&\quad - 2(b-\sqrt{\mu\alpha})\gamma\, \E\big[ \langle g_n,  P^\bot v_n\rangle\big] - 2\sqrt\mu\gamma \,\E\big[
		\langle g_n,  P^\bot x'_n\rangle\big] + \gamma^2\,\E\big[\| P^\bot g_n\|^2\big]\\
		&= (b-\sqrt{\mu\alpha})^2\, \E\big[\| P^\bot v_n\|^2\big] + 2\sqrt\mu(b-\sqrt{\mu\alpha}) \,\E\big[ \langle  P^\bot v_n,  P^\bot x'_n\rangle\big] + \mu\E\big[\| P^\bot x'_n\|^2\big]\\
		&\quad - 2(b-\sqrt{\mu\alpha})\gamma\, \E\big[ \langle \nabla f(x_n'),  P^\bot v_n\rangle\big] - 2\sqrt\mu\gamma \,\E\big[
		\langle \nabla f(x_n'),  P^\bot x'_n\rangle\big] + \gamma^2\,\E\big[\| P^\bot g_n\|^2\big].
	\end{align*}
	
	{\bf Step 3.} We now consider the third term
	\begin{align*}
		\lambda\,\E\left[ \|  P v_{n+1} \|^2 \right] &= \lambda \rho^2 \,\E\left[ \|  P (v_n - \sqrt{\alpha} g_n )  \|^2\right]\\
		&= \lambda\rho^2\,\E\left[ \| P v_n\|^2 + \alpha\| P g_n\|^2 - 2\sqrt{\alpha}  \langle g_n,  P v_n\rangle \right]\\
		&= \lambda\rho^2\,\E\left[ \| P v_n\|^2 + \alpha\| P g_n\|^2 - 2\sqrt{\alpha}  \langle \nabla f(x_n'),  P v_n\rangle \right].
	\end{align*}
	
	{\bf Step 4.} We now add the estimates of steps 2 and 3, with 
	\[
	\lambda = \frac{(b-\sqrt{\mu\alpha})\gamma}{\rho^2\sqrt\alpha } = \frac{(b+\sqrt{\mu\alpha})^2}{b-\sqrt{\mu\alpha}}\,\frac{\gamma}{\sqrt\alpha}
	\]
	such that the coefficients $-2(b-\sqrt{\mu\alpha})\gamma$ of $\E[\langle \nabla f(x_n'),  P^\bot v_n\rangle]$ and $-2\lambda \rho^2\sqrt\alpha$ of $\E[\langle \nabla f(x_n'),  P v_n\rangle]$ coincide. Note that in the deterministic case $\gamma = b\sqrt\alpha = \sqrt\alpha$ and we recover the coefficient chosen in the proof of Theorem \ref{theorem discrete}.
	\begin{align*}
		\E&\left[\|b P^\bot v_{n+1} + \sqrt\mu P^\bot x'_{n+1}\|^2  + \lambda \|  P v_{n+1} \|^2 \right]\\
		&= (b-\sqrt{\mu\alpha})^2\, \E\big[\| P^\bot v_n\|^2\big] + 2\sqrt\mu(b-\sqrt{\mu\alpha}) \,\E\big[ \langle  P^\bot v_n,  P^\bot x'_n\rangle\big] + \mu\E\big[\| P^\bot x'_n\|^2\big]\\
		&\quad - 2(b-\sqrt{\mu\alpha})\gamma\, \E\big[ \langle \nabla f(x_n'), v_n\rangle\big] - 2\sqrt\mu\gamma \,\E\big[\langle \nabla f(x_n'),  P^\bot x'_n\rangle\big]+ \frac{(b-\sqrt{\mu\alpha})\gamma}{\sqrt\alpha}\,\E\big[ \| P v_n\|^2\big] \\
		& \quad + \gamma^2\,\E\big[\| P^\bot g_n\|^2\big] + (b-\sqrt{\mu\alpha})\gamma\sqrt{\alpha}\, \E\big[\| P g_n\|^2\big].
	\end{align*}
	We note that the coefficient of the norm of the tangential gradient is $(b-\sqrt{\mu\alpha})\sqrt\alpha\,\gamma\leq \gamma^2$ by the definition of $\gamma = (b-\sqrt{\mu\alpha})\sqrt\alpha+\sqrt\mu\eta$.
	Next, we combine this estimate with the bound on $f(x_{n+1})$ from Step 1 and use the geometric conditions \eqref{eq assumption convexity wrt minimizer} and \eqref{eq assumption mild non-convexity} on $f$ to control the inner products of $\nabla f(x_n')$ with $v_n$ and $ P^\bot x_n'$ in the previous expression as well as the variance bound \eqref{eq variance bound} for the gradient estimates:
	\begin{align*}
		\L_{n+1} 
		&\leq \E\big[ f(x'_n) - \inf f\big] -  \frac\eta2\,\E\big[\|\nabla f(x'_n)\|^2\big] + \frac{\gamma^2}2 \E\big[\|g_n\|^2\big] + \frac{L\eta^2}2\,\sigma_a^2\\
		&\quad +\frac{(b-\sqrt{\mu\alpha})^2}2\, \E\big[\| P^\bot v_n\|^2\big] + \sqrt\mu(b-\sqrt{\mu\alpha}) \,\E\big[\langle  P^\bot  v_n,  P^\bot x'_n\rangle\big] + \frac{\mu}2\E\big[\| P^\bot x'_n\|^2\big]\\
		&\quad - (b-\sqrt{\mu\alpha})\frac{\gamma}{\sqrt\alpha}\,\E\left[f(x_n') - f(x_n) - \frac{\eps\alpha}2\| v_n\|^2\right]\\
		&\quad - \sqrt\mu\gamma \,\E\left[f(x_n') - \inf f + \frac\mu2 \,\| P^\bot x_n'\|^2\right] + \frac{(b-\sqrt{\mu\alpha})\gamma}{2\sqrt\alpha}\,\E\big[ \| P v_n\|^2\big]\\
		&= \left(1- (b-\sqrt{\mu\alpha})\frac{\gamma}{\sqrt\alpha}-\sqrt\mu\gamma\right) \E\big[f(x_n')\big] -(1-\sqrt\mu\gamma)\inf f + (b-\sqrt {\mu\alpha})\frac{\gamma}{\sqrt\alpha}\,\E[f(x_n)]\\
		&\quad + \frac{\gamma^2(1+\sigma_m^2)-\eta}2\,\E\big[\|\nabla f(x_n')\|^2\big]+ \frac{L\eta^2 + \gamma^2}2\sigma_a^2 \\
		&\quad + \frac{(b-\sqrt{\mu\alpha})^2+(b-\sqrt{\mu\alpha})\eps\gamma\sqrt{\alpha}}2\,\E\big[\| P v_n\|^2\big] + \sqrt\mu(b-\sqrt{\mu\alpha}) \E\big[\langle  P^\bot x_n',  P^\bot v_n\rangle\big]\\
		&\quad + \frac{\mu - \sqrt\mu\gamma\mu}2\,\E\big[\| P^\bot x_n'\|^2] + \frac{(b-\sqrt{\mu\alpha})\gamma(1+\eps\alpha)}{2\sqrt\alpha}\,\E\big[\| P v_n\|^2\big].
	\end{align*}
	
	{\bf Step 5.} Recall that
	\[
	\alpha = \frac{1- \sqrt{\mu\eta(1+\sigma_m^2)}}{1-\sqrt{\mu(1+\sigma_m^2)\eta} + \sigma_m^2}\,\eta, \qquad b = \sqrt{\frac{(1+\sigma_m^2)\alpha}\eta}
	\]
	and thus
	\[
	\rho = \frac{b-\sqrt{\mu\alpha}}{b+\sqrt{\mu\alpha}} = \frac{\sqrt{(1+\sigma_m^2)\alpha/\eta} - \sqrt{\mu\alpha}}{\sqrt{(1+\sigma_m^2)\alpha/\eta}+\sqrt{\mu\alpha}} = \frac{1 - \sqrt{\frac{\mu\eta}{1+\sigma_m^2}}} {1 + \sqrt{\frac{\mu\eta}{1+\sigma_m^2}}} 
	\]
	as desired.
		
	Let us verify that $\gamma = \sqrt\alpha/b$. This is equivalent to 
	\[
	\sqrt\mu (\eta-\alpha) + \sqrt{\frac{1+\sigma_m^2}\eta}\alpha - \sqrt{\frac\eta{1+\sigma_m^2}} = \left(\sqrt{\frac{1+\sigma_m^2}\eta} -\sqrt\mu\right)\alpha + \left(\sqrt\mu\eta - \sqrt{\frac\eta{1+\sigma_m^2}}\right) = 0
	\]
	i.e.\ to the choice
	\[
	\alpha = \frac{\sqrt\mu\eta - \sqrt{\frac\eta{1+\sigma_m^2}}}{\sqrt\mu - \sqrt{\frac{1+\sigma_m^2}\eta}} =  \frac{\sqrt{\mu\eta} - \frac1{\sqrt{1+\sigma_m^2}}} { \sqrt{\mu\eta} - \sqrt{1+\sigma_m^2}}\,\eta = \frac{1- \sqrt{\mu\eta(1+\sigma_m^2)}}{1+\sigma_m^2 - \sqrt{\mu\eta(1+\sigma_m^2)}}\,\eta
	\]
	which we made above.
	In particular, we find that
	\[
	(b-\sqrt{\mu\alpha})\frac{\gamma}{\sqrt\alpha} + \sqrt\mu\gamma
	= \left(\frac b{\sqrt\alpha}-\sqrt\mu + \sqrt\mu\right)\,\gamma
	= \frac{b\gamma}{\sqrt\alpha} = 1
	\]
	and therefore
	\begin{align*}
		&\left(1- (b-\sqrt{\mu\alpha})\frac{\gamma}{\sqrt\alpha}-\sqrt\mu\gamma\right) \E\big[f(x_n')\big] +(1-\sqrt\mu\gamma)\inf f + (b- \sqrt{\mu\alpha})\frac{\gamma}{\sqrt\alpha}\,\E[f(x_n)]\\
		&\qquad= (1-\sqrt\mu\gamma)\,\E\big[f(x_n) - \inf f\big].
	\end{align*}
	In the coefficient of $\E\big[\|\nabla f(x_n')\|^2$, we have the cancellations
	\[
	(1+\sigma_m^2)\gamma^2 - \eta = (1+\sigma_m^2)\frac{\alpha}{b^2} -\eta = \frac{\eta}{(1+\sigma_m^2)\alpha}\,\alpha - \eta = 0.
	\]
	By the same analysis, the coefficient of additive noise is
	\[
	L\eta^2 + \gamma^2 = L\eta^2 + \frac{\eta}{1+\sigma_m^2},
	\]
	so overall
	\begin{align*}
		\L_{n+1} &\leq (1-\sqrt\mu\,\gamma) \,\E\big[f(x_n) - \inf f\big] + \frac{L\eta^2 + \frac{\eta}{1+\sigma_m^2}}2\,\sigma_a^2\\
		&\quad + \frac{(b-\sqrt{\mu\alpha})^2+(b-\sqrt{\mu\alpha})\eps\gamma\sqrt{\alpha}}2\,\E\big[\| P v_n\|^2\big] + \sqrt\mu(b-\sqrt{\mu\alpha}) \E\big[\langle  P^\bot x_n',  P^\bot v_n\rangle\big]\\
		&\quad + \frac{a^2 - \sqrt\mu\gamma\mu}2\,\E\big[\| P^\bot x_n'\|^2] + \frac{(b-\sqrt{\mu\alpha})\gamma(1+\eps\sqrt\alpha)}{2\sqrt\alpha}\,\E\big[\| P v_n\|^2\big].
	\end{align*}
	We now analyze the terms corresponding to the quadratic terms. Using $\gamma = \sqrt\alpha/b$, we obtain
	\begin{align*}
		\frac{(b-\sqrt{\mu\alpha})^2 + (b-\sqrt{\mu\alpha})\eps\gamma\sqrt{\alpha}}{b^2}
		&\leq \left(1-\frac{\sqrt{\mu\alpha}} {b}\right) \left(1+\frac{(\eps\gamma-\sqrt\mu)\sqrt\alpha}b\right)\\
		&= \left(1-\sqrt\mu\,\gamma\right)\left(1+\frac{(\eps\gamma-\sqrt\mu)\sqrt\alpha}b\right)\\
		&\leq 1-\sqrt\mu\,\gamma
	\end{align*}
	for the coefficient of $\E[\|v_n\|^2]$ if $\eps\gamma\leq a$, i.e.\ if
	$\eps\leq\frac{\sqrt\mu}\gamma = \sqrt{\frac{\mu(1+\sigma_m^2)}{\eta}}$ 
	Analogously, we see that the coefficient of $\E\big[\langle  P^\bot v_n, P^\bot x_n'\rangle\big]$ that 
	\[
	\sqrt\mu(b-\sqrt{\mu\alpha}) = \sqrt\mu b \,\frac{b-\sqrt{\mu\alpha}}b = \sqrt\mu b\left(1-\sqrt\mu\,\frac{\sqrt\alpha}b\right) = \sqrt\mu b(1-\sqrt\mu\,\gamma)
	\]
	and for the coefficient of $\E\big[\| P^\bot x_n'\|^2\big]$ that
	\[
	\frac{\mu - \sqrt\mu\gamma\mu}{\mu} = 1 - \frac{\gamma\mu}{\sqrt\mu} = 1- \sqrt\mu\gamma.
	\]
	Before proceeding to the next term, we observe that
	\[
	\gamma^2 = \frac{\eta}{1+\sigma^2} \geq \eta\frac{1-\sqrt{\mu\eta(1+\sigma^2)}}{1+\sigma^2 - \sqrt{\mu\eta(1+\sigma^2)}} = \alpha,
	\] and hence $\eps\leq \sqrt\mu/\gamma \leq \sqrt\mu\gamma/\alpha.$
	Finally, we note for the coefficient of $\E[\| P v_n\|^2]$ that
	\begin{align*}
		\frac{(b-\sqrt{\mu\alpha})\gamma (1+\eps{\alpha})}{\sqrt\alpha\,\lambda}
		&= \frac{(b-\sqrt{\mu\alpha})\gamma (1+\eps\alpha)(b-\sqrt{\mu\alpha})\sqrt\alpha}{\sqrt\alpha\,\gamma(b+\sqrt{\mu\alpha})^2}\\
		& = \left(\frac{b-\sqrt{\mu\alpha}}{b+\sqrt{\mu\alpha}}\right)^2(1+\eps\alpha) \\
		&\leq \left(\frac{b-\sqrt{\mu\alpha}}{b}\right)^2(1+\eps\alpha)\\
		&= \left(1- \sqrt\mu\gamma\right)^2(1+\eps\alpha)\\
		&\leq \left(1- \sqrt\mu\gamma\right)^2(1+ \sqrt\mu\gamma)\\
		&= (1-\sqrt\mu\gamma)\big(1- \mu\gamma^2)\\
		&\leq 1-\sqrt\mu\gamma
	\end{align*}
	as desired.
	Overall, we find that
	\begin{align*}
		\L_{n+1} &\leq (1-\sqrt\mu\gamma)\bigg\{\E[f(x_n)- \inf f] + \frac{b^2}2\,\E\big[\| P^\bot v_n\|^2\big] + ab \,\E\big[\langle  P^\bot v_n, P^\bot x_n'\rangle\big] \\
		&\quad + \frac{\mu}2\,\E\big[\| P^\bot x_n'\|^2\big] + \frac\lambda2 \,\E\big[\| P v_n\|^2\big]\bigg\} + \frac{L\eta^2 + \frac\eta{1+\sigma_m^2}}2\,\sigma_a^2\\
		&=  (1-\sqrt\mu\gamma)\bigg\{\E\big[f(x_n) - \inf f\big] + \frac12 \,\E\big[\| P^\bot v_n +  P^\bot x_n'\|^2\big] + \frac\lambda2\,\E\big[\| P v_n\|^2\big]\bigg\}\\
		&\quad + \frac{L\eta^2 + \frac{\eta}{1+\sigma_m^2}}2\,\sigma_a^2\\
		&= (1-\sqrt\mu\gamma)\,\L_n + \frac{L\eta^2 + \frac{\eta}{1+\sigma_m^2}}2\,\sigma_a^2.
	\end{align*}
	By Lemma \ref{lemma recursion}, we deduce
	\[
	\L_n \leq (1-\sqrt\mu\gamma)^n \L_0 + \frac{L\eta^2 + \frac{\eta}{1+\sigma_m^2}}{2\sqrt\mu\gamma}\,\sigma_a^2 = \frac{L(1+\sigma_m^2)\eta^2 + \eta}{2(1+\sigma_m^2)\sqrt\mu\,\sqrt{\eta/(1+\sigma_m^2)}}.
	\]
	Since $L(1+\sigma_m^2)\eta \leq 1$, we can simplify the noise term to
	\[
	\frac{L(1+\sigma_m^2)\eta^2 + \eta}{2(1+\sigma_m^2)\sqrt\mu\,\sqrt{\eta/(1+\sigma_m^2)}}\leq \frac{\sigma_a^2\sqrt\eta}{\sqrt{\mu(1+\sigma_m^2)}}.\qedhere
	\]
\end{proof}

\section{A brief comparison of geometric conditions for optimization}\label{appendix convexity and pl}

\subsection{Definitions and elementary properties}
In this section, we compare some common geometric assumptions in optimization theory. Recall the following notions.

\begin{definition}
    Let $U\subseteq\R^d$ be an open set.
    We say that a $C^1$-function $f:U\to\R$
    \begin{enumerate}
        \item is {\em $\gamma$-quasar convex} if $\argmin f \neq \emptyset$ and if the inequality
        \[
        \langle \nabla f(x), x-x^*\rangle  \geq \gamma\big(f(x) - f(x^*)\big)
        \]
        holds for any $x\in U$ and any $x^*\in \argmin f$.

        \item star-convex if it is $1$-quasar convex.

        \item $(\gamma,\mu)$-strongly quasar convex if $\argmin f \neq \emptyset$ and
        \[
        \langle \nabla f(x), x-x^*\rangle  \geq \gamma\left(f(x) - f(x^*) + \frac\mu2\,\|x-x^*\|^2\right)        
        \]
        \item {\em is first order $\mu$-strongly convex} if 
        \[
        f(x) \geq f(z) + \langle \nabla f(z), x-z\rangle + \frac\mu2\|x-z\|^2 \qquad\forall\ x, z\in U.
        \]

        \item {\em satisfies the first order $\mu$-strong aiming condition} if for all $x$ we have
        \[
        f(\pi(x)) \geq f(x) + \langle \nabla f(x), \pi(x)-x\rangle + \frac\mu2\|x-\pi(x)\|^2 \qquad\forall\ x \in U
        \]
        where
        \[
        \pi(x) = \argmin\left\{\|x-z\|^2 : f(z) = \inf_{x'\in U}f(x')\right\}.
        \]
        In particular, we assume that the set of minimizers of $f$ is non-empty and that there exists a unique closest point $\pi(x)$ for all $x\in U$.

        \item {\em satisfies a PL condition with PL constant $\mu$} if 
        \[
        \|\nabla f(x)\|^2 \geq 2\mu \big(f(x) - \inf f\big).
        \]
    \end{enumerate}
\end{definition}

If $U$ is a convex set, the fourth condition is of course equivalent to regular $\mu$-strong convexity.

\begin{lemma}\label{lemma geometric implications}
    \begin{enumerate}
        \item If $f$ is first order $\mu$-strongly convex and has a minimizer in $U$, then $f$ satisfies the $\mu$-strong aiming condition.
        
        \item If $f$ satisfies the $\mu$-strong aiming condition, then $f$ satisfies the PL condition with the same constant $\mu$.

        \item If $f$ is $\mu$-strongly aiming on $\mathcal U_\alpha = \{x\in\R^d : f(x)<\alpha\}$, then the line segment connecting $x$ and $\pi(x)$ is contained in $\mathcal U_\alpha$ for all $x\in\mathcal U_\alpha$.

        \item If $f$ is $L$-Lipschitz continuous on $\mathcal U_\alpha$ and satisfies the PL-inequality with constant $\mu$ on $\mathcal U_\alpha$, then $\mu\leq L$.

        \item If $f$ is $(\gamma,\mu)$-strongly quasar convex, then $\argmin f$ consists of a single point.

        \item If $f$ is $\gamma$-quasar convex and $U$ is star-shaped with respect to the minimizer $x^*\in \argmin f$, then all sub-level sets of $f$ are star-shaped with respect $x^*$. In particular, the set of minimizers is convex.
    \end{enumerate}
\end{lemma}

If $U=\R^d$, then any strongly convex function $f:U\to\R$ has a minimizer in $U$. On general open sets, this is not guaranteed. 

\begin{proof}
{\bf First claim.}
    If $f$ is first order $\mu$-strongly convex and has a minimizer in $U$, then the minimizer $x^*$ is unique since for $x\neq x^*$ we have
    \[
    f(x) \geq f(x^*) + \langle \nabla f(x^*), x-x^*\rangle + \frac\mu2 \,\|x-x^*\|^2 = f(x^*) + \frac\mu2 \,\|x-x^*\|^2 > f(x^*).
    \]
    In particular, for every $x$ there exists a unique closest minimizer $\pi(x) = x^*$. The strong aiming condition therefore requires the first order convexity condition only for pairs of points $x, x^*$ rather than all points $x, z$. 

{\bf Second claim.} This result follows by the same proof that implies the PL condition for strongly convex functions: If $f$ satisfies the $\mu$-strong aiming condition, then
\[
\langle \nabla f(x),\,x-\pi(x)\rangle \geq f(x) - f(\pi(x)) + \frac\mu2\,\|x-\pi(x)\|^2
\]
and thus
\begin{align*}
\|\nabla f(x)\| &\geq \left\langle \nabla f(x),\,\frac{x-\pi(x)}{\|x-\pi(x)\|}\right\rangle\\
    &\geq \frac{f(x) - f(\pi(x))}{\|x-\pi(x)\|} + \frac\mu2\,\|x-\pi(x)\|\\
    &\geq \min_{\xi>0} \left(\frac{f(x) - f(\pi(x))}\xi + \frac\mu2\,\xi\right).
\end{align*}
Setting the derivative with respect to $\xi$ to zero, we find that the minimum is achieved when
\[
\frac{f(x) - f(\pi(x))}{\xi^2} = \frac\mu2, \quad\text{so }\xi = \sqrt{\frac{2\big(f(x) - \inf f\big)}\mu},
\]
so
\[
\|\nabla f(x)\| \geq \sqrt{\frac{\mu\big(f(x) - \inf f\big)}{2}} + \frac\mu2 \sqrt{\frac2\mu\,\big(f(x) - \inf f\big)} = \sqrt{2\mu\,\big(f(x)-\inf f\big)}.
\]
The PL condition follows by squaring both sides.

{\bf Third claim.} Let $x_t = (1-t)\pi(x) + tx$ for $0\leq  t \leq 1$.  First we observe that $\pi(x_t) = \pi(x)$ for every $t\in[0,1]$. Indeed, if there is another minimizer $z$ of $f$ such that $\|x_t- z\|\leq \|x_t,\pi(x)\|$ then 
\[
\|x-\pi(x)\| = \|x-x_t\|+\|x_t-\pi(x)\| \leq  \|x-x_t\| + \|x_t-z\| \leq \|x-z\|
\]
since $x,x_t,$ and $\pi(x)$ all lie on a straight line. If there exists a unique closest point $\pi(x)$ in $\mathcal M$, then we find that $z= \pi(x)$.

We note that $x_t-\pi(x) = t(x-\pi(x))$ and thus
\[
\frac{d}{dt} f(x_t) = \langle \nabla f(x_t), x-\pi(x)\rangle = \frac1t \langle \nabla f(x_t), \:x_t- \pi(x_t)\rangle \geq 0.
\]
In particular, $t\mapsto f(x_t)$ is an increasing function on the set $I_x:= \{t>0 : x_t\in \mathcal U_\alpha\}$. If $I_x$ has multiple connected components in $(0,1)$, this is only possible if $f=\alpha$ on the boundaries. If $f=\alpha$ on the lower boundary of a connected component, then $f\geq \alpha$ inside the entire interval as $f$ increases, contradicting the definition of $\mathcal U_\alpha$.

{\bf Fourth claim.}
Since $f$ is continuous, $\mathcal U_\alpha$ is open. Therefore, $x_t:= x_t - t\nabla f(x)\in \mathcal U_\alpha$ if $t$ is small. 
If $f$ is $L$-Lipschitz continuous in $\mathcal U_\alpha$, then
\[
f(x_t) \leq f(x) + \langle f(x), \:x_t-x\rangle + \frac{L}2 \|x_t-x\|^2
\]
for all $t$ such that $x_s\in \mathcal U_\alpha$ for $s\in(0,t)$ and $t\leq 1/L$ by Lemma \ref{lemma gd}. Since the function $t\mapsto f(x) - \frac t2\|\nabla f(x)\|^2$ is decreasing in $t$, we see that $x_t\in \mathcal U_\alpha$ for $t\in [0, 1/L]$. In particular, we note that
\begin{align*}
0 &\leq f(x- \nabla f/L) - \inf f \leq f(x) - \frac1{2L}\,\|\nabla f(x)\|^2 - \inf f \leq f(x) - \frac{2\mu}{2L}\big(f(x)-\inf f\big) - \inf f \\
&\leq (1-\mu/L) \big(f(x) - \inf f\big),
\end{align*}
implying the result.

{\bf Fifth claim.} Assume that $x^*, x'\in \argmin f$. Define $x_t = t x^* + (1-t)x'$. Then
\[
\frac{d}{dt}f(x_t) = \langle \nabla f(x_t), x^* - x'\rangle = \left\langle \nabla f(x_t), \frac{x_t - x'}t\right\rangle\geq \frac{\gamma}t\left(f(x_t) - \min f + \frac\mu2 \|x_t-x'\|^2\right)>0
\]
unless $\|x_t - x'\|^2 = 0$, i.e.\ unless $x^*=x'$. Thus $f(x_t)$ is strictly monotone increasing on $\{t\in[0,1] : x_t \in U\}$. Since $x^* \in U$, this means that $f$ must be strictly increasing on the final segment $(1-\xi, 1]$ where it reaches the global minimizer $x^*$ at $t=1$, leading to a contradiction.

{\bf Sixth claim.} This follows by essentially the same argument as the fifth claim: If $x$ is any point, then $tx + (1-t)x^*\in U$ since $U$ is star-shaped about $x^*$ and
\[
\frac{d}{dt} f(x_t) = \langle \nabla f(x_t), x-x^*\rangle =  \left\langle \nabla f(x_t), \frac{x_t-x^*}t\right\rangle \geq \frac\gamma t\left(f(x_t) - f(x^*)\right)\geq 0.
\]
In particular, $f$ is increasing along any rays starting at $x^*$, so if $f(x) < \alpha$, then $f(x_t)<\alpha$ for any $t\in[0,1]$.

The set of minimizers is star-shaped about every minimizer, hence convex.
\end{proof}

\subsection{A one-dimensional example}\label{appendix 1d example}
In the following simple one-dimensional example, we illustrate the hierarchy of geometric conditions between convexity and the PL condition.

\begin{example}\label{example 1d returns}
Let $R>0$ and $\eps\in(0,1)$. Consider the even function given for $x>0$ by
\begin{align*}
    f(x) &= \frac{1+\eps \sin (2R\log x)}2\,x^2\\
    f'(x) &= \big(1+\eps \sin(2R\log x) + R\eps\,\cos(2R\log x)\big) x\\
    f''(x) &= 1+\eps \sin(2R\log x) + R\eps\,\cos(2R\log x) + 2R\eps\,\cos(2R\log x) - 2R^2\eps\,\sin(2R\log x)\\
    &= 1 + \eps(1-2R^2)\,\sin(2R\log x) + 3R\eps\,\cos(2R\log x).
\end{align*}
We first note that evidently $\frac{1-\eps}2\,x^2 \leq f(x) \leq \frac{1+\eps}2 \,x^2$ for all $x$. In particular, $x^*=0$ is the unique global minimizer of $f$.

{\bf $L$-smoothness.} The function $g(\xi) = A\sin (\xi)+ B\cos(\xi)$ attains its maximum when $A\cos \xi - B\sin\xi=0$ for $A, B\in\R$, i.e. $\sin\xi = \frac{A}{\sqrt{A^2+B^2}}$ and $\cos \xi = \frac{B}{\sqrt{A^2+B^2}}$. In particular $\max_\xi g(\xi) = \sqrt{A^2+B^2}$, so
\[
|f''(x)| \leq 1 + \eps\sqrt{\big(1-2R^2\big)^2 + (3R)^2} = 1+ \eps\sqrt{1+5R^2 + 4R^4}.
\]
The second derivative is discontinuous at $x^*=0$, but bounded on $\R\setminus\{0\}$ by $L=1 + \eps\sqrt{1+5R^2 + 4R^4}$, i.e.\ $f$ is $L$-smooth.

{\bf Convexity.} By the same consideration, we see that $f$ is convex if and only if $f''\geq 0$, i.e.\ if and only if $\eps\sqrt{1+5R^2 + 4R^4}\leq 1$. If the inequality is strict, $f$ is strongly convex with constant $\mu = 1-\eps\sqrt{1+5R^2 + 4R^4}$.

{\bf PL inequality.} 
By the same argument as above, we see that
\[
|f'(x)|^2 \geq \big(1+ \eps\sin + R\eps\,\cos\big)^2 x^2 \geq \left(1 - \eps \sqrt{1+R^2}\right)^2x^2
\]
if $\eps \sqrt{1+R^2}<1$, where all trigonometric functions are evaluated at $\xi = 2R\log x$. 

In particular, if $\eps\sqrt{1+R^2}<1$, then
\[
|f'(x)|^2 \geq  \left(1 - \eps \sqrt{1+R^2}\right)^2x^2 \geq 2\,\frac{ \left(1 - \eps \sqrt{1+R^2}\right)^2}{1+\eps}\,\frac{1+\eps}2\,x^2 \geq 2\frac{ \left(1 - \eps \sqrt{1+R^2}\right)^2}{1+\eps}\,f(x),
\]
i.e.\ $f$ satisfies the PL condition with constant
\[
\mu = \frac{(1-\eps\sqrt{1+R^2})^2}{1+\eps}.
\]

{\bf Infinite number of local minimizers cluster at the orgin.} If $\eps\sqrt{1+R^2}>1$ on the other hand, then by the same argument $f'$ changes sign an infinite number of times in any neighborhood of the origin since $\lim_{x\to 0^+}\log x = -\infty$.

{\bf $\mu$-strong aiming condition.} 
Note that
\begin{align*}
    x\cdot f'(x) - f(x) &= \left(\frac12 + \frac\eps 2\sin(2R\log x) + R\eps\,\cos(\log x)\right)x^2\\
        &\geq \left(\frac12 - \sqrt{(\eps/2)^2 + (R\eps)^2}\right)x^2\\
        &= \frac12\left(1 - \eps\sqrt{1+4R^2}\right)x^2.
\end{align*}
In particular, $f$ is $\mu$-strongly aiming (with respect to the unique global minimizer) with
\[
\mu =1 - \eps\sqrt{1+4R^2}
\]
if $\eps\sqrt{1+4R^2}<1$ and it fails to be $\mu$-strongly aiming for any $\mu>0$ otherwise.

{\bf Quasar-convexity.} Since the minimizer is unique, $(\gamma,\mu)$-strong convexity is a strictly more general notion than strong aiming condition as we have an additional parameter $\gamma$ to relax the requirements. Essentially the same calculation reads 
\begin{align*}
x\cdot f'(x) - \gamma \,f(x) 
    &= \left(1- \frac{\gamma}2 + \eps\left((1-\gamma/2)\,\sin(\dots) + R\,\cos(\dots)\right)\right)x^2\\
    &\geq \left(1 - \frac\gamma2  - \eps\, \sqrt{(1-\gamma/2)^2 + R^2}\right)x^2.
\end{align*}
In particular, $f$ is $\gamma$-quasar convex if
\begin{align*}
1-\frac\gamma2 - \eps \sqrt{(1-\gamma/2)^2 + R^2} \geq 0 \qquad &\LRa \quad 1 - \eps\sqrt{1 + \left(\frac{R}{1-\gamma/2}\right)^2} \geq 0\\
    &\LRa\quad 1 + \left(\frac{R}{1-\gamma/2}\right)^2 \leq \frac1{\eps^2}.
\end{align*}
Such a $\gamma$ can be found if and only if $R^2 < \frac{1-\eps^2}{\eps^2}$. Choosing $\gamma$ slightly smaller, it is then also always possible to make $f$ $(\gamma',\mu')$-strongly convex for some $\gamma'\in(0,\gamma)$ and $\mu'>0$.

{\bf Relationship between conditions.} We quickly summarize the various parameter ranges for which the function $f$ satisfies good geometric conditions.

\vspace{1ex}

\begin{center}
    \renewcommand{\arraystretch}{1.6}
    \begin{tabular}{r|ccc}
        &PL condition& $\mu$-strongly aiming &$\mu$-strongly convex \hspace{2mm}\\\hline 
        Must be $<1$ & $\eps \sqrt{1+R^2}$ &  $\eps \sqrt{1+4R^2}$ & $\eps\sqrt{1+5R^2 + 4R^4}$\\ 
        Constant & $\frac{(1-\eps \sqrt{1+R^2})^2}{1+\eps}$ & $1- \eps\sqrt{1+4R^2}$ & $1-\eps\sqrt{1+5R^2 + 4R^4}$
    \end{tabular}
\end{center}

\vspace{1ex}

Evidently, classical strong convexity implies strong aiming condition with respect to the global minimizer which in turn implies the PL condition. The parameter ranges and constants are generally vastly different. All estimates, except for the PL constant, are sharp. In particular, the one-dimensional examples demonstrate that strong aiming condition is strictly weaker than convexity along the line segment connecting $x$ to $\pi(x)$.

\end{example}

\begin{figure}
    \centering
    \includegraphics[width=.33\textwidth]{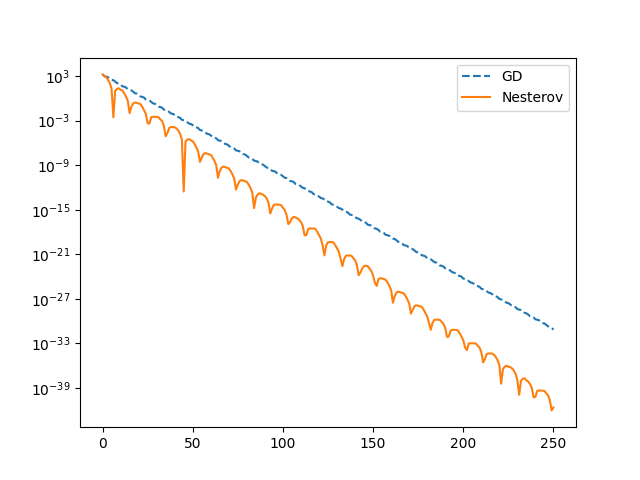}\hfill
    \includegraphics[width=.33\textwidth]{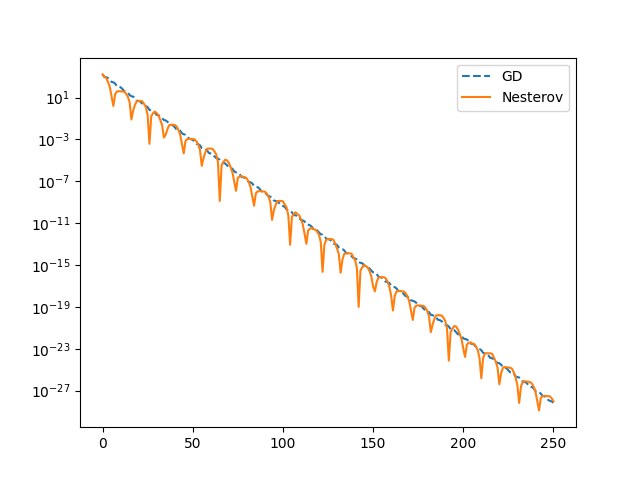}\hfill
    \includegraphics[width=.33\textwidth]{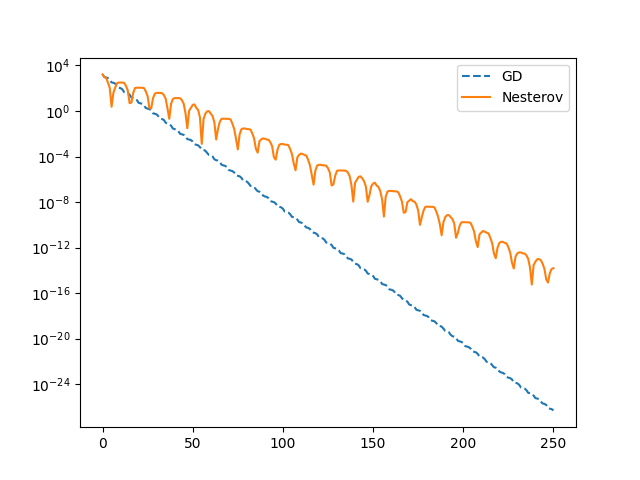}\hfill
    \caption{\label{figure 1d example 3}
    We compare the trajectories of gradient descent and Nesterov's algorithm for the objective function $f$ in Example \ref{example 1d returns} with $R=6$ and $\eps = 0.075$ (left), $\eps = 0.08$ (middle) and $\eps = 0.085$ (right). Evidently, if $\eps\sqrt{1+4R^2}$ is very close to the threshold value 1, gradient descent outperforms Nesterov's algorithm with the theoretically guaranteed parameters.
    }
\end{figure}

In the case of quadratic objective functions or functions modelled on them (such as the distance function from a manifold), the PL constant is essentially the same as the parameter of strong convexity. For the function $f$ in Example \ref{example 1d returns} on the other hand, the PL constant is noticeably larger than the parameter of strong convexity with respect to the unique minimizer. In Figure \ref{figure 1d example 3}, we observe that with the parameters $\eta=1/L$ and $\rho =(1-\sqrt{\mu\eta})/(1+\sqrt{\mu\eta})$, gradient descent may at times converge faster, at least with the parameter choice that is derived over a large function class rather than for an individual objective function.

\subsection{Deep learning}\label{appendix deep learning}

While some notions of a `good' geometry are weaker than others (for instance, strong convexity implies the PL condition but not vice versa), they all share an important common feature: {\em If $x$ is a critical point of $f$, then it is a global minimizer.} Namely, the PL inequality implies that $\|\nabla f(x)\|^2 >0$ unless $f(x) = \inf f$ and $\gamma$-quasar convexity implies that $\rangle\nabla f(x), x-x^*\rangle >0$ unless $f(x) = f(x^*) = \inf f$, showing that $\nabla f(x) \neq 0$.

In deep learning applications, critical points are guaranteed to occur under very general circumstances: If 
\[
f(\beta, a, W, b; x) = \beta + \sum_{i=1}^na_i\sigma(w_i^Tx+b)
\]
is a neural network with a single hidden layer and a $C^1$-activation function satisfying $\sigma(0) = 0$ (e.g. $\tanh$), then the loss function
\[
L(\beta, a, W, b) = \frac1n\sum_{j=1}^n \big|f(\beta, a, W, b; x_j) - y_j\big|^2,
\]
satisfies 
\[
\nabla L(\beta, a, W, b) = 0
\]
for $a = b = 0\in \R^n$, $W = 0 \in \R^{n\times d}$ and $\beta = \frac1n \sum_{j=1}^n y_j$. The same is true if a row $w_i$ of $W$ is merely orthogonal to all data points $x_j$ but does not vanish.

For deeper networks, the set of critical points becomes larger: As long as the parameters of two layers are all zero, the remaining layers can be chosen arbitrarily. If more than two layers are all zero, then the also the second parameter derivative vanishes. In particular, the critical point for which all parameters are zero cannot be a strict saddle point.

While there are guarantees that individual algorithms escape certain types of critical points almost surely (e.g. strict saddles, \citep{lee2019first, o2019behavior}), they may take very long to do so \citep{du2017gradient}. The analysis of accelerated rates becomes asymptotic at best. We claim that our notion of strong aiming condition suffers from the same `optimism' globally, but locally captures two important features of deep learning landscapes close to the set of global minimizers which are not captured by concepts which require a geometric condition with respect to {\em all} minimizers: A manifold along which we can move tangentially, and convexity in directions which are perpendicular to the manifold.

\end{document}